\documentclass[12pt,oneside,english]{amsart}
\usepackage{lmodern}
\usepackage[T1]{fontenc}
\usepackage{color}
\usepackage{geometry}
\usepackage{babel}
\usepackage{amstext}
\usepackage{amsthm}
\usepackage{amssymb}
\usepackage{graphicx}
\usepackage[unicode=true,pdfusetitle,
 bookmarks=true,bookmarksnumbered=false,bookmarksopen=false,
 breaklinks=false,pdfborder={0 0 1},backref=false,colorlinks=false]
 {hyperref}
\usepackage{etoolbox}

\makeatletter
\patchcmd{\@setauthors}{\MakeUppercase}{\large}{}{} 
\numberwithin{equation}{section}
\numberwithin{figure}{section}
\theoremstyle{plain}
\newtheorem{thm}{\protect\theoremname}
  \theoremstyle{remark}
  \newtheorem{rem}[thm]{\protect\remarkname}
  \theoremstyle{plain}
  \newtheorem{lem}[thm]{\protect\lemmaname}

\usepackage{lineno}
\usepackage[]{algorithm2e}

\@ifundefined{showcaptionsetup}{}{\PassOptionsToPackage{caption=false}{subfig}}
\usepackage{subfig}
\makeatother

\providecommand{\lemmaname}{Lemma}
\providecommand{\remarkname}{Remark}
\providecommand{\theoremname}{Theorem}

\begin{document}

\title{Application of the cut-off projection to solve a backward heat conduction problem in a two-slab composite system}
\author{Nguyen Huy Tuan}

\address{Institute for Computational Science and Technology, Ho Chi Minh city, Vietnam}

\email{thnguyen2683@gmail.com}
\author{Vo Anh Khoa$^\dagger$}

\thanks{$^\dagger$ Corresponding author}

\address{Mathematics and Computer Science Division, Gran Sasso Science Institute,
L'Aquila, Italy}
\address{Institute for Numerical and Applied Mathematics, University of Goettingen, Germany}

\email{vakhoa.hcmus@gmail.com}

\author{Mai Thanh Nhat Truong}

\address{Department of Electrical, Electronic, and Control Engineering, Hankyong
National University, Gyeonggi, Republic of Korea}

\email{mtntruong@toantin.org}

\author{Tran The Hung}

\address{Department of Mathematics, University of Hamburg, Hamburg, Germany}

\email{thehung.tran@mathmods.eu}

\author{Mach Nguyet Minh}

\address{Department of Mathematics and Statistics, University of Helsinki,
Finland}

\email{minh.mach@helsinki.fi}


\begin{abstract}
The main goal of this paper is applying the cut-off projection for solving
one-dimensional backward heat conduction problem in a two-slab
system with a perfect contact. In a constructive manner, we commence
by demonstrating the Fourier-based solution that contains the drastic
growth due to the high-frequency nature of the Fourier series. Such instability leads to the need of studying the projection method
where the cut-off approach is derived consistently. In the theoretical
framework, the first two objectives are to construct the regularized
problem and prove its stability for each noise level. Our second interest
is estimating the error in $L^2$-norm. Another supplementary
objective is computing
the eigen-elements. All in all, this paper can be considered as a preliminary attempt to solve
the heating/cooling of a two-slab composite system backward in time. Several numerical tests are provided
to corroborate the qualitative analysis.
\end{abstract}

\maketitle

\section{Introduction}

The backward-in-time heat conduction problem has been scrutinized for a long time by many approaches with fast-adapting and effectively distinct
strategies. This problem plays a critical role in different fields
of study, for example, enabling us to determine the initial temperature of a volcano
when the eruption has already started, or recover noisy, blurry digital images acquired
by camera sensors.

The need for finding numerical solutions of various kinds of such equations is originated from its natural instability due to the drastic
growth of the solutions. Once inaccurate final data are applied to the
partial differential equation (PDE), one could obtain a totally wrong solution even if the final data are measured
as accurately as possible. In fact, the
measured data in real-life applications are always inaccurate. Thus, studies on regularization methods have attracted a considerable amount of attention, where stable approximate solutions
were developed year by year. Several comprehensive materials
are given in the textbook of Latt\`es and Lions \cite{LL67}, and papers such as \cite{BC73,CO94,DB05,HDS08,Sei96,TT08,LEI98,JLR12,HDL10,Lesnic2002}.

In this paper, we consider the heat transfer problem in a two slabs composite
$\left(-b<x<0\right)$ and $\left(0<x<a\right)$. We also follow the notations in \cite{MV10}.
In particular, the problem, in which the time interval is defined by
$J:=\left[t_{0},t_{f}\right]$ for $t_{0}\ge0,t_{f}>0$, can be expressed
as:
\begin{equation}
\rho_{b}c_{b}\frac{\partial T_{b}}{\partial t}=K_{b}\frac{\partial^{2}T_{b}}{\partial x^{2}},\quad-b<x<0,\ t\in J,\label{eq:pt1}
\end{equation}
\begin{equation}
\rho_{a}c_{a}\frac{\partial T_{a}}{\partial t}=K_{a}\frac{\partial^{2}T_{a}}{\partial x^{2}},\quad0<x<a,\ t\in J.\label{eq:pt2}
\end{equation}

The boundary conditions are:
\begin{equation}
\frac{\partial T_{b}}{\partial x}=0\quad\mbox{at}\;x=-b,\quad\frac{\partial T_{a}}{\partial x}=0\quad\mbox{at}\;x=a, \ t\in J\label{eq:bien1}
\end{equation}
\begin{equation}
T_{b}=T_{a},\quad K_{b}\frac{\partial T_{b}}{\partial x}=K_{a}\frac{\partial T_{a}}{\partial x}\quad\mbox{at}\;x=0, \ t\in J.\label{eq:bien2}
\end{equation}

The model \eqref{eq:pt1}-\eqref{eq:bien2} involves a number
of dimensionless quantities, namely the temperature $T_{\alpha}$ ($^{\circ}\text{C}$), mass density $\rho_{\alpha}$ (g\slash$\text{cm}^{3}$), specific heat $c_{\alpha}$ (cal\slash$(\text{g} \cdot ^{\circ}\text{C})$), and thermal conductivity $K_{\alpha}$ (cal\slash$(\text{cm}\cdot\text{sec}\cdot ^{\circ}\text{C})$). Here the index $\alpha$ indicates either material $a$ or $b$, i.e. $\alpha \in \{a,b\}$. 
The perfect contact in \eqref{eq:bien2} indicates that the temperature and heat flux are continuous at the interface.

The model \eqref{eq:pt1}-\eqref{eq:bien2} is supplemented with the terminal
measured data $T_{\alpha}^{\varepsilon}\left(x,t_{f}\right)$ which
satisfy:
\begin{equation}
\left\Vert T_{a}\left(\cdot,t_{f}\right)-T_{a}^{\varepsilon}\left(\cdot,t_{f}\right)\right\Vert _{L^{2}\left(0,a\right)}+\left\Vert T_{b}\left(\cdot,t_{f}\right)-T_{b}^{\varepsilon}\left(\cdot,t_{f}\right)\right\Vert _{L^{2}\left(-b,0\right)}\le\varepsilon,\label{terminalcondition}
\end{equation}
where $0<\varepsilon\ll1$ represents  the maximum bound
of errors arising in measurement. Henceforward, our task is to identify
the initial value $T_{\alpha}^{\varepsilon}\left(x,t_{0}\right), -b\leq x \leq a$.

Our main objective is to prove the ill-posedness of the aforementioned
model and then deal with it by using the cut-off projection method. Historically, there
have been several works that treat ill-posed problems using this method. For instances, Daripa and his co-workers in \cite{Ternat12,Ternat11}
contributed significantly. In particular, the cut-off method \cite{Ternat12}
was used to control the spurious effects on the solution due to short
wave components of the round-off and truncation errors. Recently, the cut-off method was utilized by Regi\'nska
and Regi\'nski \cite{Regin06} and Tuan et al. \cite{TKMT17} to stabilize the $\cosh\left(x\right)$-
and $\sinh\left(x\right)/x$-like growth in a Cauchy problem for the
3D Helmholtz equation. 




This paper is fivefold. Section \ref{sec:2} is devoted to the eigen-elements
and ill-posedness of the model. We compute the approximate
eigenvalues and eigenfunctions, along with providing the temperature solution by Fourier-based
modes. In Section \ref{sec:3}, the cut-off method is described by providing
 the regularized problem corresponding
to the terminal measured data $T_{\alpha}^{\varepsilon}\left(x,t_{f}\right)$. We also show the error estimates which clearly imply the stability of the regularized solution at
each noise level. Lastly, some numerical tests are given in Section
\ref{sec:4}, whilst Section \ref{sec:5} is dedicated to a few conclusions
and pointing out problems for possible further development. 

\section{Eigen-elements\label{sec:2}}

In the context of Fourier representations, the eigen-elements
including the eigenvalues and eigenfunctions can be obtained by solving
the Sturm\textendash Liouville problem. Such problem would
be of the form:
\begin{equation}
\phi_{\alpha}''+\lambda_{\alpha}^{2}\phi_{\alpha}=0\quad\mbox{for}\;\alpha\in\left\{ a,b\right\} ,\label{eq:SLprobl}
\end{equation}
associated with the boundary conditions \eqref{eq:bien1} and \eqref{eq:bien2}.
After some straightforward calculations and using \eqref{eq:bien1}
one has
\begin{equation}
\phi_{b} (x)=\theta_{b}\cos\left(\lambda_{b}\left(x+b\right)\right),\quad\phi_{a} (x)=\theta_{a}\cos\left(\lambda_{a}\left(x-a\right)\right),\label{eq:expression}
\end{equation}
where $\theta_a,\theta_b$ are constants.

Applying the transmission conditions \eqref{eq:bien2} at the interface $x=0$,
we arrive at
\begin{equation}
\theta_{b}\cos\left(\lambda_{b}b\right)-\theta_{a}\cos\left(\lambda_{a}a\right)=0,\label{eq:in1}
\end{equation}
\begin{equation}
\theta_{b}K_{b}\lambda_{b}\sin\left(\lambda_{b}b\right)+\theta_{a}K_{a}\lambda_{a}\sin\left(\lambda_{a}a\right)=0,\label{eq:in2}
\end{equation}
which lead to the transcendental equation for eigenvalues:
\begin{equation}
K_{b}\lambda_{b}\sin\left(\lambda_{b}b\right)\cos\left(\lambda_{a}a\right)+K_{a}\lambda_{a}\sin\left(\lambda_{a}a\right)\cos\left(\lambda_{b}b\right)=0.\label{eq:canon}
\end{equation}

For brevity, we denote the thermal diffusivities ($\text{cm}^{2}\slash\text{sec}$) by
\[
\kappa_{\alpha}=\frac{K_{\alpha}}{\rho_{\alpha}c_{\alpha}}\quad\mbox{for}\;\alpha\in\left\{ a,b\right\} .
\]
Since we need the same time dependence in both slabs, the condition
\begin{equation}
\kappa_{b}\lambda_{b}^{2}=\kappa_{a}\lambda_{a}^{2}\label{eq:con1}
\end{equation}
 is required. Hereby, the above equation enables us to derive
the new eigenvalue condition from \eqref{eq:canon}:
\begin{equation}
\frac{K_{b}}{\sqrt{\kappa_{b}}}\sin\left(\lambda_{b}b\right)\cos\left(\lambda_{a}a\right)+\frac{K_{a}}{\sqrt{\kappa_{a}}}\sin\left(\lambda_{a}a\right)\cos\left(\lambda_{b}b\right)=0.\label{eq:canon1}
\end{equation}

To sum up, the two transcendental equations \eqref{eq:con1} and \eqref{eq:canon1}
allow us to compute the eigenvalues $\lambda_{an},\lambda_{bn}$ for $n\in\mathbb{N}$.

\subsection*{A computational approach}

The equations \eqref{eq:con1} and \eqref{eq:canon1} spontaneously generate a
 system of nonlinear algebraic equations with $2N$ variables $\lambda_{bn}$
and $\lambda_{an}$ for $n\in\left\{ 0,...,N\right\} $. From the
numerical point of view, the common way is to use the Newton method
which is known as the one-step iterative scheme with quadratic convergence
\cite{OR70} for solving equations of the form $\mbox{F}\left(\mbox{x}\right)=0$.
Besides, so far there have been several effective approaches such as the
Newton-Simpson's method of \cite{CT07} and the improved Newton's method by the Closed-Open quadrature formula of \cite{NW09}. However, there would
be some basic impediments if we exploit such approaches. Noticeably,
the transcendental system \eqref{eq:con1} and \eqref{eq:canon1} has
infinitely many roots and furthermore, we cannot ensure that the obtained solutions correspond
one-to-one to the indexes $n$, which means that the result may be missed
out or exceeded. Thus, by varying initial guesses $\mbox{x}_{0}$, 
it would produce a few confusing results. For instance, taking $K_{\alpha}=\kappa_{\alpha}=1$,
we easily obtain the family of solutions
\begin{equation}
\lambda_{\alpha k}=\frac{k\pi}{a+b}\quad\mbox{for}\;k\in\mathbb{N},\label{eq:abc-1}
\end{equation}
then with $a=3,b=5$ and the two initial points
\[
\mbox{x}_{0}:=\begin{bmatrix}0 & 1 & \cdots & N-1 & N\\
0 & 1 & \cdots & N-1 & N
\end{bmatrix}^{\mbox{T}}\;\mbox{and}\;\mbox{x}_{1}:=0.5\mbox{x}_{0},
\]
the eigenvalues governed by the standard Newton method (for $N=50$)
are showed in Figure \ref{fig:fsolve1} for comparison. It is apparent that there are many differences between the two figures. Basically,
the actual final eigenvalue with $N=50$ is approximately $19.63495$
from \eqref{eq:abc-1}, while the others are, respectively, $49.87278$
and $26.70353$ in Figure \ref{fig:fsolve1}. The results also indicate that
with the guess $\mbox{x}_{0}$ there are 48 admissible values
obtained by the elimination of equal values whilst there are only 42 values for the
guess $\mbox{x}_{1}$.

In order to overcome this issue, we apply the built-in function \texttt{fzero} of MATLAB to several different initial points over some pre-defined
range. Thanks to the structure of the transcendental equations \eqref{eq:con1} and \eqref{eq:canon1}, we
are capable of reducing the system to the following equation:
\begin{eqnarray}\label{eqn:lambda_b}
\frac{K_{b}}{\sqrt{\kappa_{b}}}\sin\left(\lambda_{b}b\right)\cos\left(\sqrt{\frac{\kappa_{b}}{\kappa_{a}}}\lambda_{b}a\right)+\frac{K_{a}}{\sqrt{\kappa_{a}}}\sin\left(\sqrt{\frac{\kappa_{b}}{\kappa_{a}}}\lambda_{b}a\right)\cos\left(\lambda_{b}b\right)=0,
\end{eqnarray}
where we have derived from \eqref{eq:con1} that
\begin{eqnarray}\label{eqn:lambda_a}
\lambda_{a}=\sqrt{\frac{\kappa_{b}}{\kappa_{a}}}\lambda_{b}
\end{eqnarray}
and substituted this to \eqref{eq:canon1}. Once $\lambda_{b}$ is found, it is straightforward to compute the corresponding
$\lambda_{a}$ from \eqref{eqn:lambda_a}. The algorithm to compute the first $N$ non-negative eigenvalues $\lambda_{bn}$ from \eqref{eqn:lambda_b} is described as follows:

\begin{algorithm}[H]
\KwData{function $F$, interval $[0,d]$, number of the eigenvalues $N$, small radius $\delta$}
\KwResult{$N+1$ roots of $F$}
found\_roots $\gets\emptyset$\;
output\_roots $\gets\emptyset$\;
\While{size(found\_roots) $< N + 1$}{
found\_roots $\gets$ all roots of $F$ on $[0,d]$\;
\eIf{size(found\_roots) $< N + 1$}{
$d \gets d + \delta$\;
}
{output\_roots $\gets$ first $N + 1$ roots from found\_roots\;
}}
return output\_roots
\caption{Compute the first $N$ non-negative eigenvalues.}
\end{algorithm}

\begin{figure}
\begin{centering}
\includegraphics[scale=0.8]{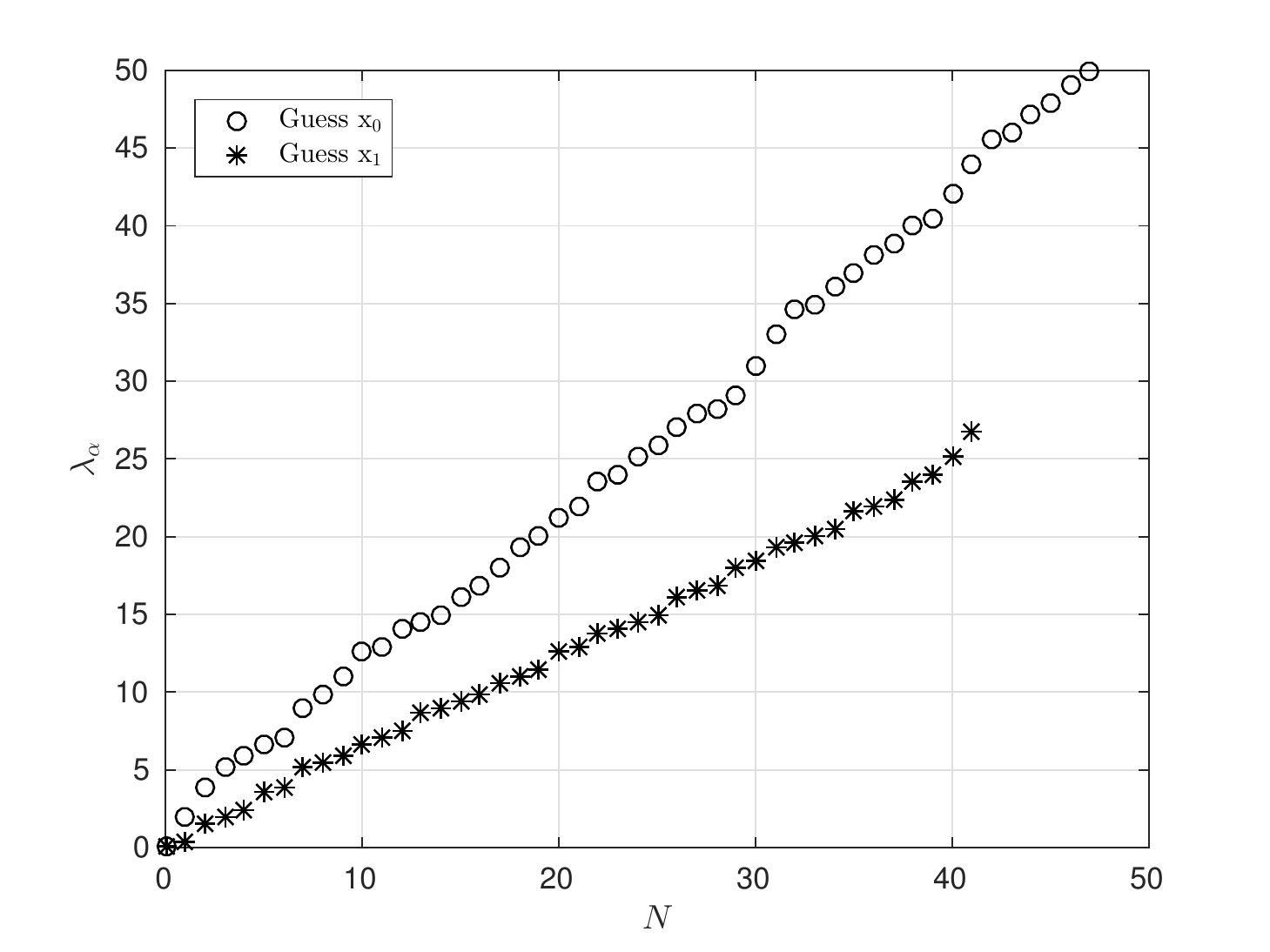}
\par\end{centering}
\caption{Numerical eigenvalues for the explicit
case \eqref{eq:abc-1} using the standard Newton method with different initial guesses $\mbox{x}_{0}$
and $\mbox{x}_{1}$ at the $10$th iteration and with the stopping
constant $\delta=10^{-10}$.\label{fig:fsolve1}}
\end{figure}

\begin{figure}
\begin{raggedright}
\subfloat[The explicit
case \eqref{eq:abc-1}\label{fig:fzero}]{\begin{centering}
\includegraphics[scale=0.6]{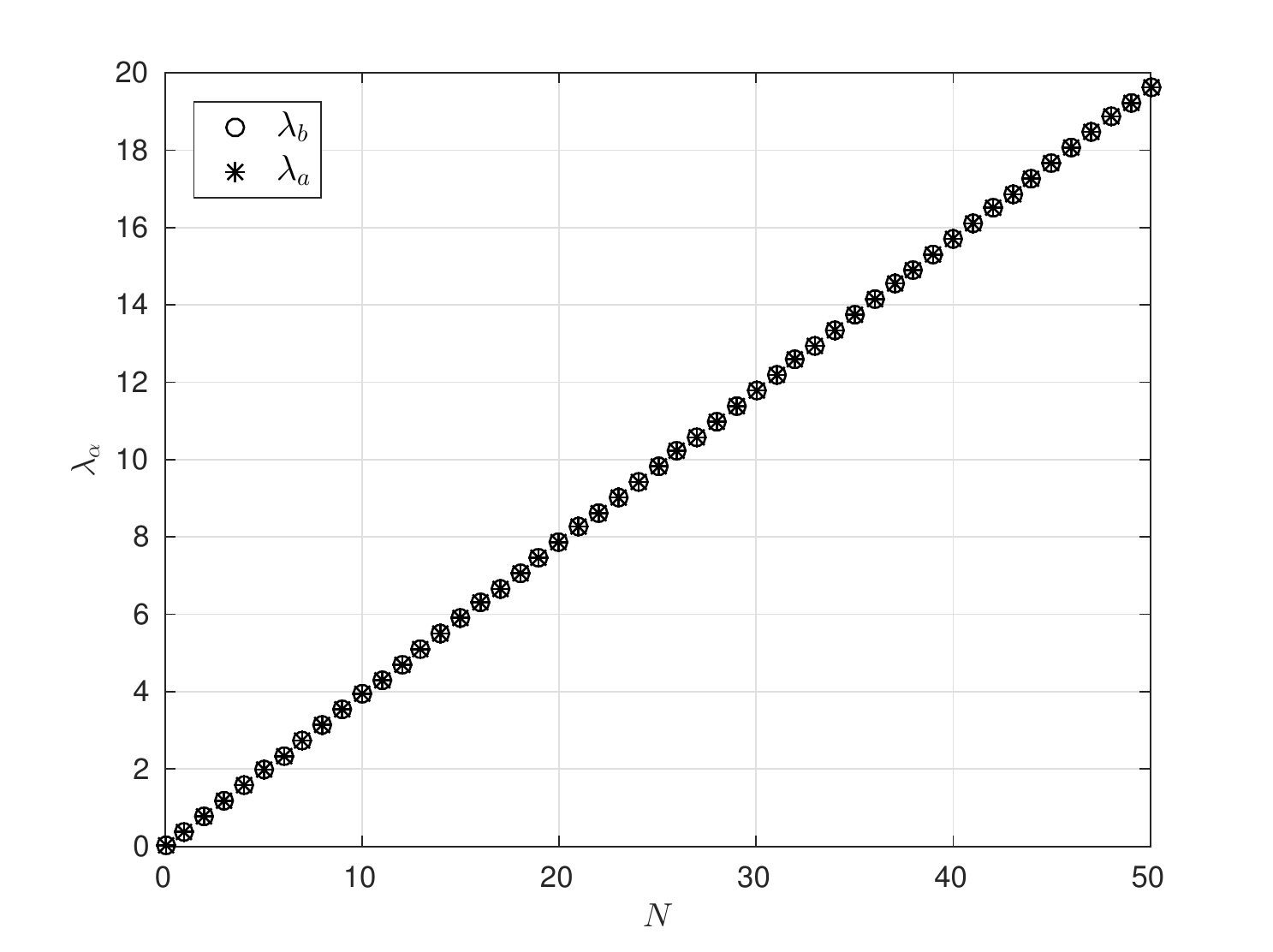}
\par\end{centering}
}\subfloat[The real
example in Remark \ref{rem:1}\label{fig:fzero1}]{\begin{centering}
\includegraphics[scale=0.6]{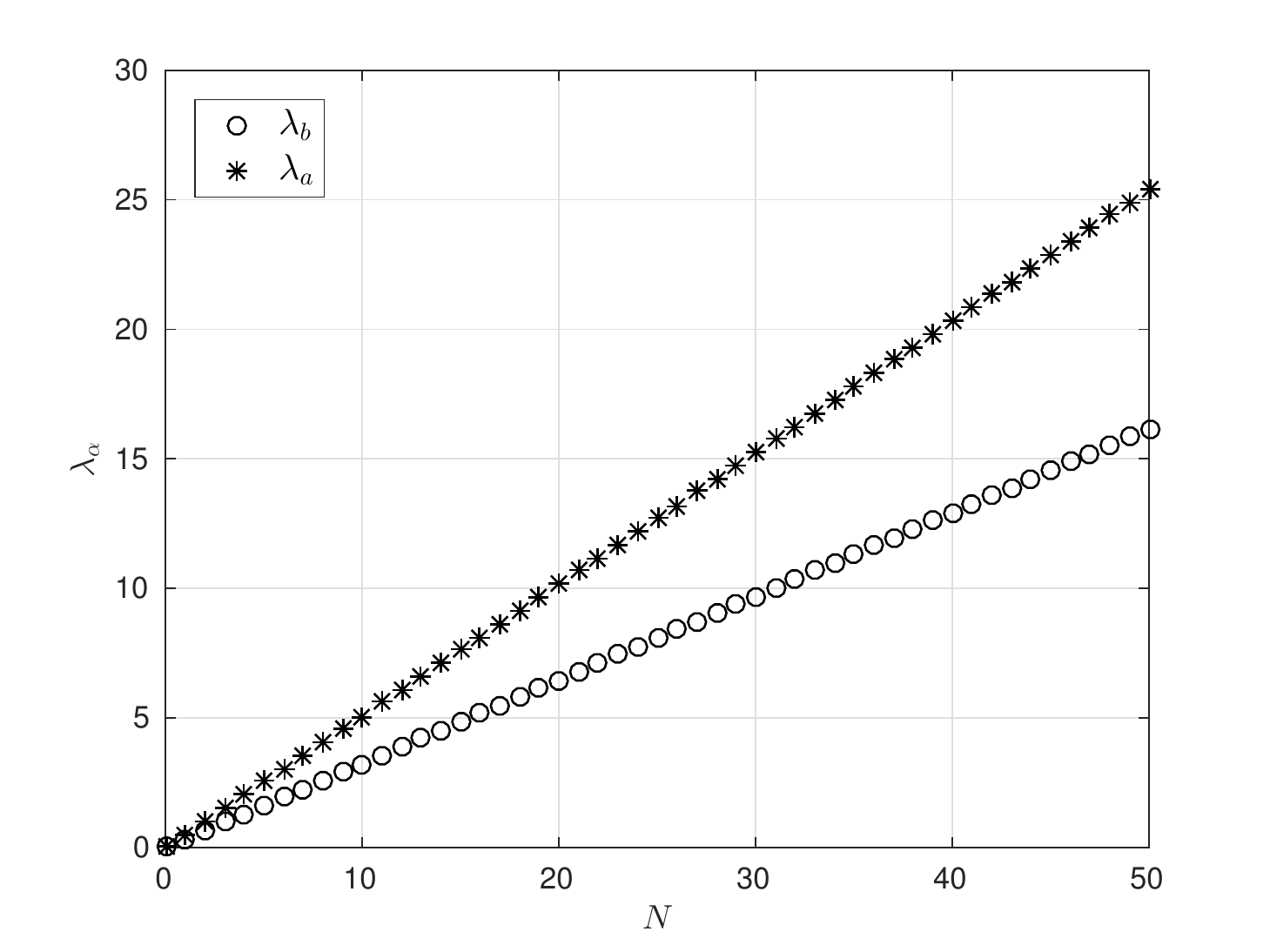}
\par\end{centering}
}
\par\end{raggedright}
\caption{Numerical eigenvalues using the MATLAB-built-in function \texttt{fzero}.\label{fig:2.2}}
\end{figure}

\begin{rem}
\label{rem:1}A complicated case is provided in Figure \ref{fig:fzero1}
where the eigenvalues cannot be solved explicitly. The materials we
illustrate are the copper versus the molybdeum, which was experienced
by Josell et al. \cite{Josell97}, providing $\kappa_{\alpha}\in\left\{ 0.838,0.339\right\} $
and $K_{\alpha}\in\left\{ 3.42,1.05\right\} $, respectively. Here,
this case is governed by the same parameters of the previous explicit
case \eqref{eq:abc-1}, i.e. $N=50,a=3,b=5$, with $d=10$ and $\delta = 10^{-3}$. By comparing Figures \ref{fig:fsolve1} and \ref{fig:2.2}, it is easy
to see the differences and how well the simple algorithm works. Especially,
the computed final numerical eigenvalue, i.e. $\lambda_{b50}=19.63495$ in Figure \ref{fig:fzero},
is very closed to the exact value. The only disadvantage of this approach
is the long execution time that may rise in higher dimensions.
\end{rem}
Now, let us consider the eigenfunctions $\phi_{\alpha}$.
By choosing appropriate constants $\theta_a$ and $\theta_b$ in \eqref{eq:expression}, we obtain the following eigenfunctions:
\begin{equation}
\tilde{\phi}_{n} (x)=\phi_{bn} (x)=\frac{\cos\left(\lambda_{bn}\left(x+b\right)\right)}{\cos\left(\lambda_{bn}b\right)}\quad\mbox{for the left slab}\;\left(-b<x<0\right),\label{eq:eigen1}
\end{equation}
and
\begin{equation}
\tilde{\phi}_{n} (x)=\phi_{an} (x)=\frac{\cos\left(\lambda_{an}\left(x-a\right)\right)}{\cos\left(\lambda_{an}a\right)}\quad\mbox{for the right slab}\;\left(0<x<a\right).\label{eq:eigen2}
\end{equation}

\begin{rem}
\label{rem:orth}For different indexes $m\ne n$, the eigenfunctions
$\tilde{\phi}_{m}$ and $\tilde{\phi}_{n}$ are orthogonal in the sense that
\begin{equation}
\frac{K_{b}}{\kappa_{b}}\left\langle \tilde{\phi}_{m},\tilde{\phi}_{n}\right\rangle _{L^{2}\left(-b,0\right)}+\frac{K_{a}}{\kappa_{a}}\left\langle \tilde{\phi}_{m},\tilde{\phi}_{n}\right\rangle _{L^{2}\left(0,a\right)}=0. \label{eq:2.13}
\end{equation}

On the other hand, we also obtain the same equality for the derivatives
of such eigenfunctions. In fact, these functions are orthogonal in the sense that
\[
K_{b}\left\langle \tilde{\phi}^{'}_{m},\tilde{\phi}^{'}_{n}\right\rangle _{L^{2}\left(-b,0\right)}+K_{a}\left\langle \tilde{\phi}^{'}_{m},\tilde{\phi}^{'}_{n}\right\rangle _{L^{2}\left(0,a\right)}=0.
\]

Let us denote
\begin{eqnarray*}
\mathcal{N}_{n} & := & \frac{K_{b}}{\kappa_{b}}\left\Vert \tilde{\phi}_{n}\right\Vert _{L^{2}\left(-b,0\right)}^{2}+\frac{K_{a}}{\kappa_{a}}\left\Vert \tilde{\phi}_{n}\right\Vert _{L^{2}\left(0,a\right)}^{2}  = \frac{bK_{b}}{2\kappa_{b}}\frac{1}{\cos^{2}\left(\lambda_{bn}b\right)}+\frac{aK_{a}}{2\kappa_{a}}\frac{1}{\cos^{2}\left(\lambda_{an}a\right)}\quad\mbox{for}\;n\in \mathbb{N}^{*},
\end{eqnarray*}
\begin{align*}
\mathcal{M}_{n} & :=K_{b}\left\Vert \tilde{\phi}^{'}_{n}\right\Vert _{L^{2}\left(-b,0\right)}^{2}+K_{a}\left\Vert \tilde{\phi}^{'}_{n}\right\Vert _{L^{2}\left(0,a\right)}^{2} =\frac{bK_{b}}{2}\frac{\lambda_{bn}^{2}}{\cos^{2}\left(\lambda_{bn}b\right)}+\frac{aK_{a}}{2}\frac{\lambda_{an}^{2}}{\cos^{2}\left(\lambda_{an}a\right)}\quad\mbox{for}\;n\in \mathbb{N}^{*},
\end{align*}
with
\[
\mathcal{N}_{0}:=\frac{K_{b}}{\kappa_{b}}\left\Vert \tilde{\phi}_{0}\right\Vert _{L^{2}\left(-b,0\right)}^{2}+\frac{K_{a}}{\kappa_{a}}\left\Vert \tilde{\phi}_{0}\right\Vert _{L^{2}\left(0,a\right)}^{2}=\frac{bK_{b}}{\kappa_{b}}+\frac{aK_{a}}{\kappa_{a}},
\]

and
\[
\mathcal{M}_{0}:=K_{b}\left\Vert \tilde{\phi}^{'}_{0}\right\Vert _{L^{2}\left(-b,0\right)}^{2}+K_{a}\left\Vert \tilde{\phi}^{'}_{0}\right\Vert _{L^{2}\left(0,a\right)}^{2}=0.
\]
\end{rem}

\section{Main results}\label{sec:3}

\subsection{Solution by Fourier-mode and Instability}\label{subsec:3-1}

From the transcendental equations \eqref{eq:con1}-\eqref{eq:canon1}
and the formulation of the eigenfunctions in \eqref{eq:eigen1}-\eqref{eq:eigen2},
we are competent to construct the temperature solution throughout
the two-slab system, denoted by $T\left(x,t\right)$. Basically, the
Fourier-mode for the solution of this system cannot be obtained in
the same manner with the classical backward heat equation due to the
non-trivial orthogonality (showed in Remark \ref{rem:orth}) in which
we cannot make equivalent computations. Nevertheless, we are still
able to compute it by following the structural property of the
solution in the classical case.

By assigning $\bar{\lambda}_{n}:=\kappa_{b}\lambda_{bn}^{2}=\kappa_{a}\lambda_{an}^{2}$,
the temperature solution (in the finite series truncated computational sense) would be of
the following form:
\begin{equation}
T_N \left(x,t\right)=\sum_{n=0}^{N} C_{n}e^{\bar{\lambda}_{n}\left(t_{f}-t\right)}\tilde{\phi}_{n}\left(x\right),\quad\left(x,t\right)\in\left[-b,a\right]\times\left[t_{0},t_{f}\right],\label{eq:tempsol}
\end{equation}
for some $N\in\mathbb{N}$ and the data associated with the final
state of temperature is provided by
\begin{equation}
T_N \left(x,t_{f}\right)=\sum_{n=0}^{N}C_{n}\tilde{\phi}_{n}\left(x\right).\label{finalstate}
\end{equation}

Before investigating the instability of the solution as well as the
cut-off projection, our first concern is to determine the Fourier-type
coefficients $C_{n}$ in \eqref{eq:tempsol} for a fixed
$N\in\mathbb{N}$. It is worth noticing that due to the two-slab system
being considered, such coefficients vary on each material, hence we
denote by $C_{\alpha n}$ for $\alpha\in\left\{ a,b\right\} $
to distinguish those different values.

According to the representation \eqref{finalstate}, the algorithm
to compute the coefficient $C_{\alpha n}$ is provided as
follows: Let $J_{\alpha}$ be the number of discrete nodes in
the material $\alpha$. Suppose these are measured 
at the time $t_{f}$, producing a set $\left\{ T_{\alpha}^{\varepsilon}\left(x_{j},t_{f}\right)\right\} _{j=\overline{0,J_{\alpha}}}$.
Then the choice of a suitable $J_{\alpha}$ will be discussed later.
We are led to the matrix-formed system $\mathbb{A}\mathbb{X}=\mathbb{B}$
where $\mathbb{A}\in\mathbb{R}^{J_{\alpha}+1}\times\mathbb{R}^{N+1}$,
$\mathbb{X}\in\mathbb{R}^{N+1}$, $\mathbb{B}\in\mathbb{R}^{J_{\alpha}+1}$
are expressed by
\begin{equation}
\mathbb{A}=\begin{bmatrix}\phi_{\alpha0}\left(x_{0}\right) & \phi_{\alpha1}\left(x_{0}\right) & \cdots & \phi_{\alpha N}\left(x_{0}\right)\\
\phi_{\alpha0}\left(x_{1}\right) & \phi_{\alpha1}\left(x_{1}\right) & \cdots & \phi_{\alpha N}\left(x_{1}\right)\\
\vdots & \vdots & \ddots & \vdots\\
\phi_{\alpha0}\left(x_{J_{\alpha}}\right) & \phi_{\alpha1}\left(x_{J_{\alpha}}\right) & \cdots & \phi_{\alpha N}\left(x_{J_{\alpha}}\right)
\end{bmatrix},\;\mathbb{X}=\begin{bmatrix}C_{\alpha0}\\
C_{\alpha1}\\
\vdots\\
C_{\alpha N}
\end{bmatrix},\;\mathbb{B}=\begin{bmatrix}T_{\alpha}^{\varepsilon}\left(x_{0},t_{f}\right)\\
T_{\alpha}^{\varepsilon}\left(x_{1},t_{f}\right)\\
\vdots\\
T_{\alpha}^{\varepsilon}\left(x_{J_{\alpha}},t_{f}\right)
\end{bmatrix}.
\label{mainsystem}
\end{equation}

\begin{rem}
Multiplying the expression \eqref{finalstate} by the $n$-th eigenfunction
$\tilde{\phi}_{n}\left(x\right)$ using the inner product in $L^{2}\left(-b,0\right)$
with the weight $K_{b}/\kappa_{b}$, then do the same using the inner product in
$L^{2}\left(0,a\right)$, we obtain
\begin{equation}
\vartheta_{b}C_{bn}+\vartheta_{a} C_{an}=\mathcal{N}_{n}^{-1}\left(\vartheta_{b}\frac{K_{b}}{\kappa_{b}}\left\langle T_{b}\left(\cdot,t_{f}\right),\phi_{bn}\right\rangle _{L^{2}\left(-b,0\right)}+\vartheta_{a}\frac{K_{a}}{\kappa_{a}}\left\langle T_{a}\left(\cdot,t_{f}\right),\phi_{an}\right\rangle _{L^{2}\left(0,a\right)}\right)\; \textrm{for }\vartheta_{\alpha}\in\mathbb{R}, \label{eq:C+C}
\end{equation}
which also guarantees that $\left|C_{\alpha n}\right|<\infty$ for
all $n\in\mathbb{N}$.
\end{rem}
The second objective concerns the instability of the temperature in this
model manifested by the expression \eqref{eq:tempsol}.
Due to the intrinsic property of the eigenvalues $\bar{\lambda}_{n}$
which reads
\[
0\le\bar{\lambda}_{1}<\bar{\lambda}_{2}<...<\bar{\lambda}_{n}<...<\lim_{n\to\infty}\bar{\lambda}_{n}=\infty,
\]
the rapid escalation is indeed catastrophic and it would be the same
issue that commonly occurs in previous works (e.g. \cite{NTT10,Regin06,TAKL2017,MEIL01,JLR11}).
Note that even if the coefficients $C_{n}$ may decrease quickly,
a large error between the Fourier-based solution \eqref{eq:tempsol}
and the exact solution always happens. Thus, any computational method
are impossible to be applied. To be precise, we give the proof of the
following theorem that the temperature through the system is exponentially
unstable.
\begin{thm}
The truncated approximate solution $T_N$ is exponentially
unstable over the $L^{2}$-norm by the following estimate:
\begin{equation}
\left\Vert T_N \left(\cdot,t\right)\right\Vert _{L^{2}\left(-b,a\right)}^{2}\geq\left(\max\left\{ \dfrac{K_{a}}{\kappa_{a}},\dfrac{K_{b}}{\kappa_{b}},1\right\} \right)^{-1}\sum_{n=0}^{N}\min\left\{ \left|C_{bn}\right|^{2},\left|C_{an}\right|^{2}\right\} e^{2\bar{\lambda}_{n}\left(t_{f}-t\right)}\mathcal{N}_{n}.\label{eq:3.5}
\end{equation}
\end{thm}
\begin{proof}
Consider the $L^{2}$-norm of \eqref{eq:tempsol} over $\left(-b,a\right)$:
\begin{align}
\left\Vert T_N \left(\cdot,t\right)\right\Vert _{L^{2}\left(-b,a\right)}^{2} & =\left\langle \sum_{n=0}^{N}C_{n}e^{\bar{\lambda}_{n}\left(t_{f}-t\right)}\tilde{\phi}_{n},\sum_{n=0}^{N}C_{n}e^{\bar{\lambda}_{n}\left(t_{f}-t\right)}\tilde{\phi}_{n}\right\rangle _{L^{2}\left(-b,a\right)}\nonumber \\
 & =\sum_{n=0}^{N}\sum_{m=0}^{N}C_{bn}C_{bm}e^{\bar{\lambda}_{n}\left(t_{f}-t\right)}e^{\bar{\lambda}_{m}\left(t_{f}-t\right)}\left\langle \phi_{bn},\phi_{bm}\right\rangle _{L^{2}\left(-b,0\right)}\nonumber \\
 & +\sum_{n=0}^{N}\sum_{m=0}^{N}C_{an}C_{am}e^{\bar{\lambda}_{n}\left(t_{f}-t\right)}e^{\bar{\lambda}_{m}\left(t_{f}-t\right)}\left\langle \phi_{an},\phi_{an}\right\rangle _{L^{2}\left(0,a\right)}\nonumber \\
 & \ge\left(\max\left\{ \dfrac{K_{a}}{\kappa_{a}},\dfrac{K_{b}}{\kappa_{b}},1\right\} \right)^{-1}\sum_{n=0}^{N}\sum_{m=0}^{N}e^{\bar{\lambda}_{n}\left(t_{f}-t\right)}e^{\bar{\lambda}_{m}\left(t_{f}-t\right)}\nonumber \\
 & \times\min\left\{ C_{bn}C_{bm},C_{an}C_{am}\right\} \left(\frac{K_{b}}{\kappa_{b}}\left\langle \phi_{bn},\phi_{bm}\right\rangle _{L^{2}\left(-b,0\right)}+\frac{K_{a}}{\kappa_{a}}\left\langle \phi_{an},\phi_{an}\right\rangle _{L^{2}\left(0,a\right)}\right).\label{3.4}
\end{align}

From the orthogonality relation \eqref{eq:2.13},
we continue to estimate \eqref{3.4} and obtain \eqref{eq:3.5}.
\end{proof}

\begin{rem}
	One needs $J_{\alpha} = N$ to ensure the unique solvability of the linear system \eqref{mainsystem}. However, the question concerning how big the constant $N$ should be taken severely hinders the choice of $J_\alpha$. Therefore, the cut-off projection that we consider in the next subsection is important not only in regularization, but also in providing a precise value of $J_\alpha$ needed at each noise level.
\end{rem}

\subsection{The cut\textendash off projection}

This part of work aims at applying the cut-off projection to deal
with the aforementioned instability as well as proposes an effective
technique to compute the temperature solution $T\left(x,t\right)$. The
idea is to construct a projection in which a suitable high frequency
$N$ is appropriately selected depending on the measurement error $\varepsilon\in\left(0,1\right)$.
In other words, for each $\varepsilon>0$ we now introduce the cut-off
mapping $\mathbb{P}_{\varepsilon}$ such that

\begin{equation}
\mathbb{P}_{\varepsilon}T(x,t)=\sum_{n\in\Theta(\varepsilon)}C_{n}e^{\bar{\lambda}_{n}\left(t_{f}-t\right)}\tilde{\phi}_{n}\left(x\right),
\label{approx}
\end{equation}
where $\Theta(\varepsilon):=\left\{ n\in\mathbb{N}:\bar{\lambda}_{n}\le N_{\varepsilon}\right\} $
is the set of admissible frequencies and $N_{\varepsilon}:=N\left(\varepsilon\right)>0$
satisfying ${\displaystyle \lim_{\varepsilon\to0}N_{\varepsilon}=\infty}$
plays the role of the regularization parameter. Henceforward, the algorithm to compute the Fourier  coefficients $C_{\alpha n}$ in Subsection \ref{subsec:3-1} is now active with $J_\alpha = N_{\varepsilon}$.

As the main objectives of this section, we prove the following
theorems to analyze the stability and the convergence of the (approximate)
regularized solution governed by the representation \eqref{approx}.
\begin{thm}
\label{thm:stable}For each $\varepsilon>0$, the regularized solution
depends continuously on the final temperature in the $L^{2}$-norm by
the following estimate:
\begin{align} \label{eq:3.8.1}
\left\Vert \mathbb{P}_{\varepsilon}T\left(\cdot,t\right)\right\Vert _{L^{2}\left(-b,a\right)}^{2}\le\min\left\{ \dfrac{K_{b}}{\kappa_{b}},\dfrac{K_{a}}{\kappa_{a}},1\right\} ^{-1}N_{\varepsilon}e^{2N_{\varepsilon}\left(t_{f}-t\right)}\left(\frac{K_{b}}{\kappa_{b}}\left\Vert T_{b}\left(\cdot,t_{f}\right)\right\Vert _{L^{2}\left(-b,0\right)}^{2} \right. \\ \nonumber
\left. +\frac{K_{a}}{\kappa_{a}}\left\Vert T_{a}\left(\cdot,t_{f}\right)\right\Vert _{L^{2}\left(0,a\right)}^{2}\right).
\end{align}
\end{thm}
\begin{proof}
Once again, we use the definition of the norm induced by inner product
in $L^{2}\left(-b,a\right)$ and then the proof is straightforward
with the non-trivial orthogonality \eqref{eq:2.13}. Indeed,
expressing the projection onto the solution gives us

\begin{align*}
& \left\Vert \mathbb{P}_{\varepsilon}T\left(\cdot,t\right)\right\Vert _{L^{2}\left(-b,a\right)}^{2} :=\left\langle \mathbb{P}_{\varepsilon}T\left(\cdot,t\right),\mathbb{P}_{\varepsilon}T\left(\cdot,t\right)\right\rangle _{L^{2}\left(-b,a\right)}\\
& =\sum_{n\in\Theta\left(\varepsilon\right)}\sum_{m\in\Theta\left(\varepsilon\right)}\bar{C}_{bn}\left(t\right)\bar{C}_{bm}\left(t\right)\left\langle \phi_{bn},\phi_{bm}\right\rangle _{L^{2}\left(-b,0\right)}
 +\sum_{n\in\Theta\left(\varepsilon\right)}\sum_{m\in\Theta\left(\varepsilon\right)}\bar{C}_{an}\left(t\right)\bar{C}_{am}\left(t\right)\left\langle \phi_{an},\phi_{am}\right\rangle _{L^{2}\left(0,a\right)},
\end{align*}
where we have denoted by $\bar{C}_{\alpha n}\left(t\right)=C_{\alpha n}e^{\bar{\lambda}_{n}\left(t_{f}-t\right)}$
for $\alpha\in\left\{ a,b\right\} $. This leads to

\begin{align} \label{3.7}
\left\Vert \mathbb{P}_{\varepsilon}T\left(\cdot,t\right)\right\Vert _{L^{2}\left(-b,a\right)}^{2}\le\min\left\{ \dfrac{K_{b}}{\kappa_{b}},\dfrac{K_{a}}{\kappa_{a}},1\right\} ^{-1}\sum_{n\in\Theta\left(\varepsilon\right)}\left(\dfrac{K_{b}\bar{C}_{bn}^{2}\left(t\right)}{\kappa_{b}}\left\Vert \phi_{bn}\right\Vert _{L^{2}\left(-b,0\right)}^{2} \right. \\
\left. +\dfrac{K_{a}\bar{C}_{an}^{2}\left(t\right)}{\kappa_{a}}\left\Vert \phi_{an}\right\Vert _{L^{2}\left(0,a\right)}^{2}\right). \nonumber
\end{align}

We then combine the elementary inequality $a_{1}a_{3}^{2}+a_{2}a_{4}^{2}\le\left(a_{1}+a_{2}\right)\left(a_{3}^2+a_{4}^2\right)$
for $a_{1},a_{2}\ge0$ and $a_{3},a_{4}\in\mathbb{R}$ with the properties
in Remark \ref{rem:1} and \eqref{eq:C+C}, to yield from \eqref{3.7}
that
\begin{align*}
\left\Vert \mathbb{P}_{\varepsilon}T\left(\cdot,t\right)\right\Vert _{L^{2}\left(-b,a\right)}^{2} & \le\min\left\{ \dfrac{K_{b}}{\kappa_{b}},\dfrac{K_{a}}{\kappa_{a}},1\right\} ^{-1}\sum_{n\in\Theta\left(\varepsilon\right)}\mathcal{N}_{n}^{-1}e^{2\bar{\lambda}_{n}\left(t_{f}-t\right)}\\
 & 
 \times \left[\left(\frac{K_{b}}{\kappa_{b}}\left|\left\langle T_{b}\left(\cdot,t_{f}\right),\phi_{bn}\right\rangle _{L^{2}\left(-b,0\right)}\right|\right)^2
+\left(\frac{K_{a}}{\kappa_{a}}\left|\left\langle T_{a}\left(\cdot,t_{f}\right),\phi_{an}\right\rangle _{L^{2}\left(0,a\right)}\right|\right)^{2}\right].
\end{align*}

Using the Cauchy-Schwarz inequality and the basic inequality $a_1^2 a_3^2+a_2^2 a_4^2\le\left(a_1^{2}+a_2^{2}\right)\left(a_3^{2}+a_4^{2}\right)$ for all $a_j\in\mathbb{R},j\in \left\{1,...,4\right\}$
we get
\begin{align*}
\left(\frac{K_{b}}{\kappa_{b}}\left|\left\langle T_{b}\left(\cdot,t_{f}\right),\phi_{bn}\right\rangle _{L^{2}\left(-b,0\right)}\right|\right)^2
+\left(\frac{K_{a}}{\kappa_{a}}\left|\left\langle T_{a}\left(\cdot,t_{f}\right),\phi_{an}\right\rangle _{L^{2}\left(0,a\right)}\right|\right)^{2} & \le\mathcal{N}_{n}\left(\frac{K_{b}}{\kappa_{b}}\left\Vert T_{b}\left(\cdot,t_{f}\right)\right\Vert _{L^{2}\left(-b,0\right)}^{2}\right.\\
 & \left.+\frac{K_{a}}{\kappa_{a}}\left\Vert T_{a}\left(\cdot,t_{f}\right)\right\Vert _{L^{2}\left(0,a\right)}^{2}\right),
\end{align*}
and by definition of the set $\Theta\left(\varepsilon\right)$, we
arrive at \eqref{eq:3.8.1}
\end{proof}
Theorem \ref{thm:stable} shows us that the cut-off projection solution
is bounded at each noise level, and such an interesting estimate proves
that the regularized solution is stable even when the final temperature
is corrupted by noise. Furthermore, its appearance enables
us to derive the $L^{2}$-type estimate between the solutions
$\mathbb{P}_{\varepsilon}T\left(x,t\right)$ and $\mathbb{P}_{\varepsilon}T^{\varepsilon}\left(x,t\right)$ (i.e. we apply the projection method to design the solution associated with the terminal measured data),
which we refer to the following lemma.
\begin{lem}
Let $\mathbb{P}_{\varepsilon}T^{\varepsilon}\left(x,t\right)$
be the regularized solution associated with the square\textendash integrable
measured data $T_{\alpha}^{\varepsilon}\left(\cdot,t_{f}\right)$
satisfying \eqref{terminalcondition}. Then the following estimate
holds:
\begin{equation}
\left\Vert \mathbb{P}_{\varepsilon}T\left(\cdot,t\right)-\mathbb{P}_{\varepsilon}T^{\varepsilon}\left(\cdot,t\right)\right\Vert _{L^{2}\left(-b,a\right)}^{2}\le\min\left\{ \dfrac{K_{b}}{\kappa_{b}},\dfrac{K_{a}}{\kappa_{a}},1\right\} ^{-1}N_{\varepsilon}e^{2N_{\varepsilon}\left(t_{f}-t\right)}\left(\frac{K_{b}}{\kappa_{b}}+\frac{K_{a}}{\kappa_{a}}\right)\varepsilon^{2}.\label{3.9}
\end{equation}
\end{lem}
\begin{proof}
Similar to the proof of Theorem \ref{thm:stable}, we arrive at
\begin{align*}
\left\Vert \mathbb{P}_{\varepsilon}T\left(\cdot,t\right) - \mathbb{P}_{\varepsilon}T^{\varepsilon} \left(\cdot,t\right) \right\Vert _{L^{2}\left(-b,a\right)}^{2} & \le\min\left\{ \dfrac{K_{b}}{\kappa_{b}},\dfrac{K_{a}}{\kappa_{a}},1\right\} ^{-1}N_{\varepsilon}e^{2N_{\varepsilon}\left(t_{f}-t\right)} \\
& \left(\frac{K_{b}}{\kappa_{b}}\left\Vert T_{b}\left(\cdot,t_{f}\right) - T^\epsilon_{b}\left(\cdot,t_{f}\right)\right\Vert _{L^{2}\left(-b,0\right)}^{2}+\frac{K_{a}}{\kappa_{a}}\left\Vert T_{a}\left(\cdot,t_{f}\right) - T^\epsilon_{a}\left(\cdot,t_{f}\right)\right\Vert _{L^{2}\left(0,a\right)}^{2}\right).
\end{align*}

The proof of the lemma is completed by applying the terminal condition \eqref{terminalcondition} to the above estimate.

\end{proof}

At this moment, it is possible for us to prove the convergence result. From the triangle inequality, it follows
\[
\left\Vert T\left(\cdot,t\right)-\mathbb{P}_{\varepsilon}T^{\varepsilon} \left(\cdot,t\right)\right\Vert _{L^{2}\left(-b,a\right)}\leq\left\Vert \mathbb{P}_{\varepsilon}T\left(\cdot,t\right)-\mathbb{P}_{\varepsilon}T^{\varepsilon} \left(\cdot,t\right)\right\Vert _{L^{2}\left(-b,a\right)}+\left\Vert T\left(\cdot,t\right)-\mathbb{P}_{\varepsilon}T\left(\cdot,t\right)\right\Vert _{L^{2}\left(-b,a\right)},
\]
which then questions us how to estimate the norm $\left\Vert T\left(\cdot,t\right)-\mathbb{P}_{\varepsilon}T\left(\cdot,t\right)\right\Vert _{L^{2}\left(-b,a\right)}$.

We remind that in the Fourier-based sense, the quantity $N$ in the
representation \eqref{eq:tempsol} tends to infinity and that, consequently,
requires the regularization parameter to satisfy ${\displaystyle \lim_{\varepsilon\to0}N_{\varepsilon}=\infty}$.
To compare the difference between $T\left(x,t\right)$ and $\mathbb{P}_{\varepsilon}T\left(x,t\right)$,
we suppose that \eqref{eq:tempsol} reaches the most ideal case, i.e.
it can be represented by
\begin{equation}
T\left(x,t\right)=\sum_{n=0}^{\infty}C_{n} e^{\bar{\lambda}_{n}\left(t_{f}-t\right)}\tilde{\phi}_{n}\left(x\right).\label{eq:newtempsol}
\end{equation}

Subtract \eqref{eq:newtempsol} and \eqref{approx}, we thus see that
\[
T\left(x,t\right)-\mathbb{P}_{\varepsilon}T\left(x,t\right)=\sum_{n\notin\Theta\left(\varepsilon\right)} C_{n}e^{\bar{\lambda}_{n}\left(t_{f}-t\right)}\tilde{\phi}_{n}\left(x\right).
\]

Therefore, very much in line with the computations in the proof of
Theorem \ref{thm:stable}, the $L^{2}$-norm of the above difference
can be upper-bounded by
\begin{align}\label{3.10}
\left\Vert T\left(\cdot,t\right)-\mathbb{P}_{\varepsilon}T\left(\cdot,t\right)\right\Vert _{L^{2}\left(-b,a\right)}^{2}\le\min\left\{ \dfrac{K_{b}}{\kappa_{b}},\dfrac{K_{a}}{\kappa_{a}},1\right\} ^{-1}\sum_{n\notin\Theta\left(\varepsilon\right)}\left(\dfrac{K_{b}\bar{C}_{bn}^{2}\left(t\right)}{\kappa_{b}}\left\Vert \phi_{bn}\right\Vert _{L^{2}\left(-b,0\right)}^{2} \right. \\
\left.+\dfrac{K_{a}\bar{C}_{an}^{2}\left(t\right)}{\kappa_{a}}\left\Vert \phi_{an}\right\Vert _{L^{2}\left(0,a\right)}^{2}\right). \nonumber
\end{align}

On the other side, we compute that
\[
\partial_x T\left(x,t\right)=\sum_{n=0}^{\infty} C_{n}e^{\bar{\lambda}_{n}\left(t_{f}-t\right)}\tilde{\phi} ^{'}\left(x\right),
\]
which gives
\begin{align*}
\left\Vert \partial_x T\left(\cdot,t\right)\right\Vert _{L^{2}\left(-b,a\right)}^{2} & =\sum_{n=0}^{\infty}\sum_{m=0}^{\infty}\bar{C}_{bn}\left(t\right)\bar{C}_{bm}\left(t\right)\left\langle \phi_{bn}^{'},\phi_{bm}^{'}\right\rangle _{L^{2}\left(-b,0\right)} +\sum_{n=0}^{\infty}\sum_{m=0}^{\infty}\bar{C}_{an}\left(t\right)\bar{C}_{am}\left(t\right)\left\langle \phi_{an}^{'},\phi_{am}^{'}\right\rangle _{L^{2}\left(0,a\right)},
\end{align*}
and using \eqref{eq:2.13} we obtain
\begin{equation}
\left\Vert \partial_x T\left(\cdot,t\right)\right\Vert _{L^{2}\left(-b,a\right)}^{2}\ge\left(\max\left\{ K_{a},K_{b},1\right\} \right)^{-1}\sum_{n=0}^{\infty}\left(K_{b}\bar{C}_{bn}^{2}\left(t\right)\left\Vert \phi_{bn}^{'}\right\Vert _{L^{2}\left(-b,0\right)}^{2}+K_{a}\bar{C}_{an}^{2}\left(t\right)\left\Vert \phi_{an}^{'}\right\Vert _{L^{2}\left(0,a\right)}^{2}\right).\label{eq:3.10}
\end{equation}

Furthermore, it reveals from \eqref{3.10} that
\begin{eqnarray}
&&\left\Vert T\left(\cdot,t\right)-\mathbb{P}_{\varepsilon}T\left(\cdot,t\right)\right\Vert _{L^{2}\left(-b,a\right)}^{2} \nonumber\\
&\le& \min\left\{ \dfrac{K_{b}}{\kappa_{b}},\dfrac{K_{a}}{\kappa_{a}},1\right\} ^{-1}\sum_{n\notin\Theta\left(\varepsilon\right)}\bar{\lambda}_{n}^{-1}\left(K_{b}\bar{C}_{bn}^{2}\left(t\right)\left\Vert \phi_{bn}^{'}\right\Vert _{L^{2}\left(-b,0\right)}^{2}+K_{a}\bar{C}_{an}^{2}\left(t\right)\left\Vert \phi_{an}^{'}\right\Vert _{L^{2}\left(0,a\right)}^{2}\right)\nonumber \\
 & \le & \min\left\{ \dfrac{K_{b}}{\kappa_{b}},\dfrac{K_{a}}{\kappa_{a}},1\right\} ^{-1}N_{\varepsilon}^{-1}\sum_{n\notin\Theta\left(\varepsilon\right)}\left(K_{b}\bar{C}_{bn}^{2}\left(t\right)\left\Vert \phi_{bn}^{'}\right\Vert _{L^{2}\left(-b,0\right)}^{2}+K_{a}\bar{C}_{an}^{2}\left(t\right)\left\Vert \phi_{an}^{'}\right\Vert _{L^{2}\left(0,a\right)}^{2}\right).\label{3.12}
\end{eqnarray}

Combining \eqref{eq:3.10} and \eqref{3.12} we observe that
\begin{equation}
\left\Vert T\left(\cdot,t\right)-\mathbb{P}_{\varepsilon}T\left(\cdot,t\right)\right\Vert _{L^{2}\left(-b,a\right)}^{2}\le\min\left\{ \dfrac{K_{b}}{\kappa_{b}},\dfrac{K_{a}}{\kappa_{a}},1\right\} ^{-1}\max\left\{ K_{a},K_{b},1\right\} N_{\varepsilon}^{-1}\left\Vert \partial_x T\left(\cdot,t\right)\right\Vert _{L^{2}\left(-b,a\right)}^{2}.\label{3.14}
\end{equation}

At this stage, we conclude that the error $\left\Vert T\left(\cdot,t\right)-\mathbb{P}_{\varepsilon}T\left(\cdot,t\right)\right\Vert _{L^{2}\left(-b,a\right)}$
will tend to zero as $\varepsilon\to0$ if the gradient of the temperature
solution is bounded over $L^{2}\left(-b,a\right)$. By using the triangle
inequality and the errors in \eqref{3.9} and \eqref{3.14}, we easily
obtain the estimate $\left\Vert T\left(\cdot,t\right)-T^{\varepsilon}\left(\cdot,t\right)\right\Vert _{L^{2}\left(-b,a\right)}$.
We thus have the following theorem.
\begin{thm}
Assume that $T\left(\cdot,t\right)\in H^{1}\left(-b,a\right)$ for all $t\in J$ and
let $\mathbb{P}_{\varepsilon}T^{\varepsilon}\left(x,t\right)$
be the regularized solution associated with the square\textendash integrable
measured data $T_{\alpha}^{\varepsilon}\left(\cdot,t_{f}\right)$.
Then the following estimate holds:
\begin{align}
\left\Vert T\left(\cdot,t\right)-\mathbb{P}_{\varepsilon}T^{\varepsilon}\left(\cdot,t\right)\right\Vert _{L^{2}\left(-b,a\right)} & \leq\min\left\{ \dfrac{K_{b}}{\kappa_{b}},\dfrac{K_{a}}{\kappa_{a}},1\right\} ^{-\frac{1}{2}}\left(\sqrt{\dfrac{K_{b}}{\kappa_{b}}}+\sqrt{\dfrac{K_{a}}{\kappa_{a}}}\right)N_{\varepsilon}^{\frac{1}{2}}e^{N_{\varepsilon}\left(t_{f}-t\right)}\varepsilon\nonumber \\
 & +\min\left\{ \dfrac{K_{b}}{\kappa_{b}},\dfrac{K_{a}}{\kappa_{a}},1\right\} ^{-\frac{1}{2}}\max\left\{ K_{a},K_{b},1\right\} ^{\frac{1}{2}}N_{\varepsilon}^{-\frac{1}{2}}\left\Vert \partial_x T\left(\cdot,t\right)\right\Vert _{L^{2}\left(-b,a\right)}.\label{3.15}
\end{align}
\end{thm}
\begin{rem}
\label{rem:8}
In general, the error estimate \eqref{3.15} can be rewritten as
\[
\left\Vert T\left(\cdot,t\right)-\mathbb{P}_{\varepsilon}T^{\varepsilon}\left(\cdot,t\right)\right\Vert _{L^{2}\left(-b,a\right)}\le C\left(N_{\varepsilon}^{\frac{1}{2}}e^{N_{\varepsilon}\left(t_{f}-t\right)}\varepsilon+N_{\varepsilon}^{-\frac{1}{2}}\right),
\]
in which $C$ is a generic positive constant. Then by choosing $N_{\varepsilon}=\beta\ln\left(\varepsilon^{-\gamma}\right)/t_{f}$
for $\beta\in\left(0,1\right)$ and $\gamma\in\left(0,1\right]$ we obtain that
\[
\left\Vert T\left(\cdot,t\right)-\mathbb{P}_{\varepsilon}T^{\varepsilon}\left(\cdot,t\right)\right\Vert _{L^{2}\left(-b,a\right)}\le C\left(\sqrt{\frac{\beta}{t_{f}}}\ln^{\frac{1}{2}}\left(\varepsilon^{-\gamma}\right)\varepsilon^{1-\gamma\left(1-\frac{t}{t_{f}}\right)\beta}+\sqrt{\frac{t_{f}}{\beta}}\ln^{-\frac{1}{2}}\left(\varepsilon^{-\gamma}\right)\right).
\]

Consequently, $\mathbb{P}_{\varepsilon}T^{\varepsilon}$ is a good approximation of $T$ as the small noise $\varepsilon$ tends to zero. On top of that, this rate of convergence is uniform in time.

It is worth pointing out that due to the structural property of the
eigen-elements, one may deduce that for $p\in\mathbb{N}$ and $m\ne n$
the functions $\tilde{\phi}_{m} ^{(2p)}$ and $\tilde{\phi}_{n} ^{(2p)}$
are orthogonal with weights $K_{b}/\kappa_{b}$ and $K_{a}/\kappa_{a}$
whilst the functions $\tilde{\phi}_{m} ^{(2p+1)}$ and $\tilde{\phi}_{n} ^{(2p+1)}$
are orthogonal with weights $K_{b}$ and $K_{a}$. On the other hand,
we see the following relations
\begin{align*}
\frac{K_{b}}{\kappa_{b}}\left\Vert \phi_{bn} ^{(2p)}\right\Vert _{L^{2}\left(-b,0\right)}^{2} =\bar{\lambda}_{n}^{-1}K_{b}\left\Vert \phi_{bn} ^{(2p+1)}\right\Vert _{L^{2}\left(-b,0\right)}^{2} &, \quad
\frac{K_{a}}{\kappa_{a}}\left\Vert \phi_{an} ^{(2p)}\right\Vert _{L^{2}\left(0,a\right)}^{2} =\bar{\lambda}_{n}^{-1}K_{a}\left\Vert \phi_{an} ^{(2p+1)}\right\Vert _{L^{2}\left(0,a\right)}^{2}.
\end{align*}

Combining these two relations with the fact that $\phi_{\alpha n} ^{(2p)}=\left(-1\right)^{p}\lambda_{\alpha n}^{2p}\phi_{\alpha n}$
leads us to the $L^{2}$-norm equivalence between $\phi_{\alpha n}$
and their high-order derivatives. Therefore, if we wish to obtain
the error estimate over the space $H^{p}\left(-b,a\right)$, the requirement
will be $T\left(\cdot,t\right)\in H^{p+1}\left(-b,a\right)$. Note that to tackle the error estimate in this context, only the difference between the high-order derivatives of the exact solution and the regularized is needed. Such
result can be found explicitly in \cite{NTT10}.
\end{rem}

\section{Numerical tests\label{sec:4}}

Following the example from Remark \ref{rem:1}, in this section we consider
three examples including the smooth, piecewise smooth and discontinuous cases
in a short-time observation $J=\left[0,0.1\right]$. Note that herein we
do not know the exact solution, but at some point we can generate artificial
data to see how the approximation works. All computations were
done using MATLAB, and only using 5 digit floating point arithmetic.
Recall that with $b=5\left(\text{cm}\right)$ and $a=3\left(\text{cm}\right)$,
the materials we illustrate herein are the copper versus the molybdeum
providing $\kappa_{\alpha}\in\left\{ 0.838,0.339\right\} $ and $K_{\alpha}\in\left\{ 3.42,1.05\right\} $,
respectively. We also select $\beta=0.05$, $\gamma=0.5$ and fix 20 discrete spatial points for each slab throughout this illustration. Accordingly, we are able to compute the regularization parameter $N_\varepsilon$ which has been chosen in Remark \ref{rem:8}.

\subsection{Example 1\label{subsec:Example-1}}
In this part, the stability of the approximate solution is investigated
to prove the approximation as $\varepsilon\to0$. Practically, we
merely have the measured final data given by
\begin{align*}
T_{b}^{\varepsilon}\left(x_{j},t_{f}\right) & =\frac{\cos\left(\lambda_{b1}\left(x_{j}+b\right)\right)}{\left(a+b\right)\cos\left(\lambda_{b1}b\right)}e^{-\frac{t_{f}}{2}}+\text{rand}\left(x_{j}\right)\varepsilon,\quad x_j \in (-b,0) ,\\
T_{a}^{\varepsilon}\left(x_{j},t_{f}\right) & =\frac{\cos\left(\lambda_{a1}\left(x_{j}-a\right)\right)}{\left(a+b\right)\cos\left(\lambda_{a1}a\right)}e^{-\frac{t_{f}}{2}}+\text{rand}\left(x_{j}\right)\varepsilon, \quad x_j \in (0,a),
\end{align*}
with $\left|\text{rand}\left(x_{j}\right)\right|\le\sqrt{2\max\left\{ a,b\right\} }$
being a random number generator.

\begin{figure}
\begin{centering}
\includegraphics[scale=0.8]{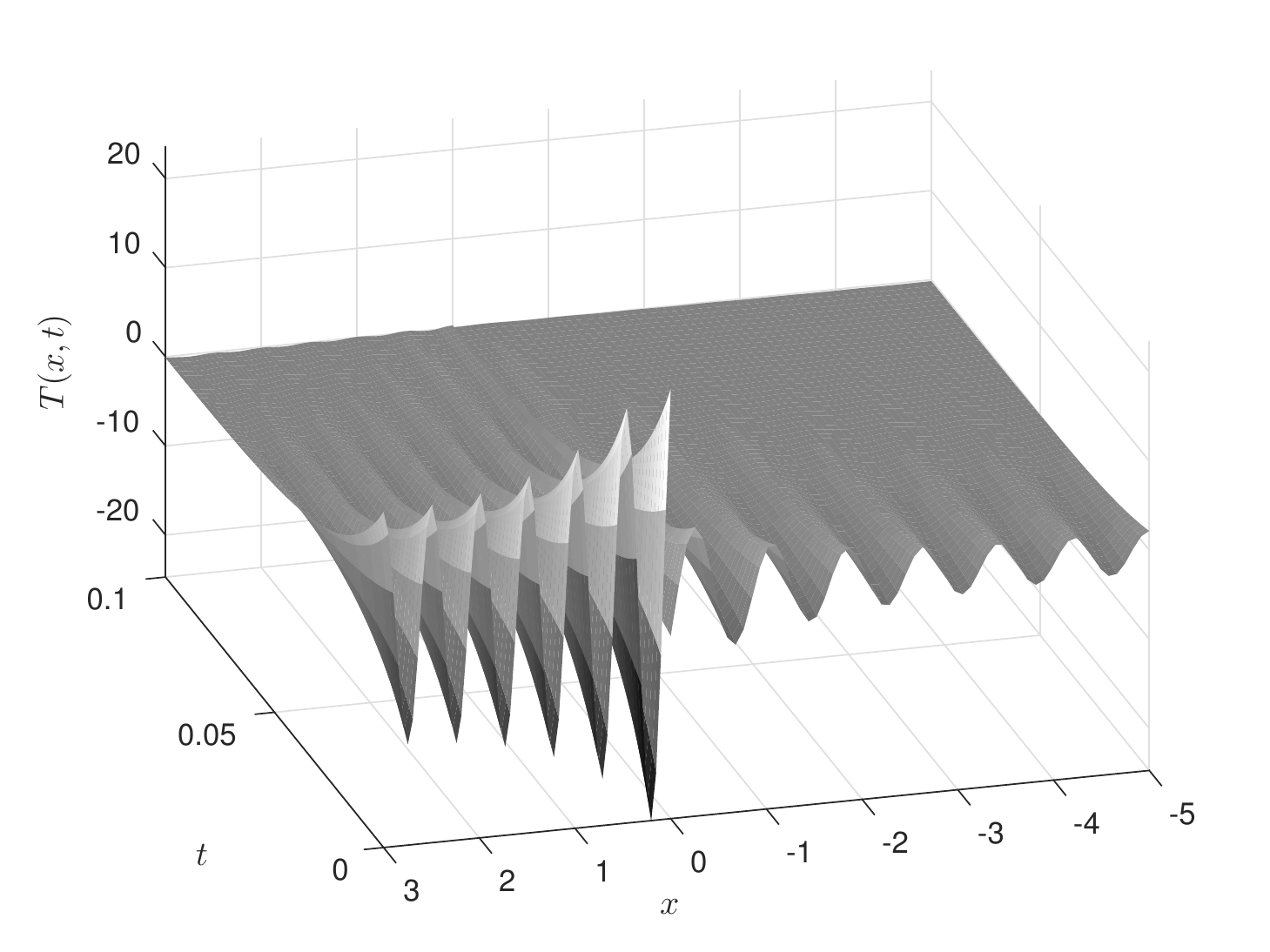}
\par\end{centering}
\caption{The instability illustration with $\varepsilon=10^{-2}$ at a cut-off
high-frequency (selected randomly) $N=50$. (Example \ref{subsec:Example-1})\label{fig:The-instability-illustration}}
\end{figure}

\begin{figure}
\begin{centering}
\includegraphics[scale=0.8]{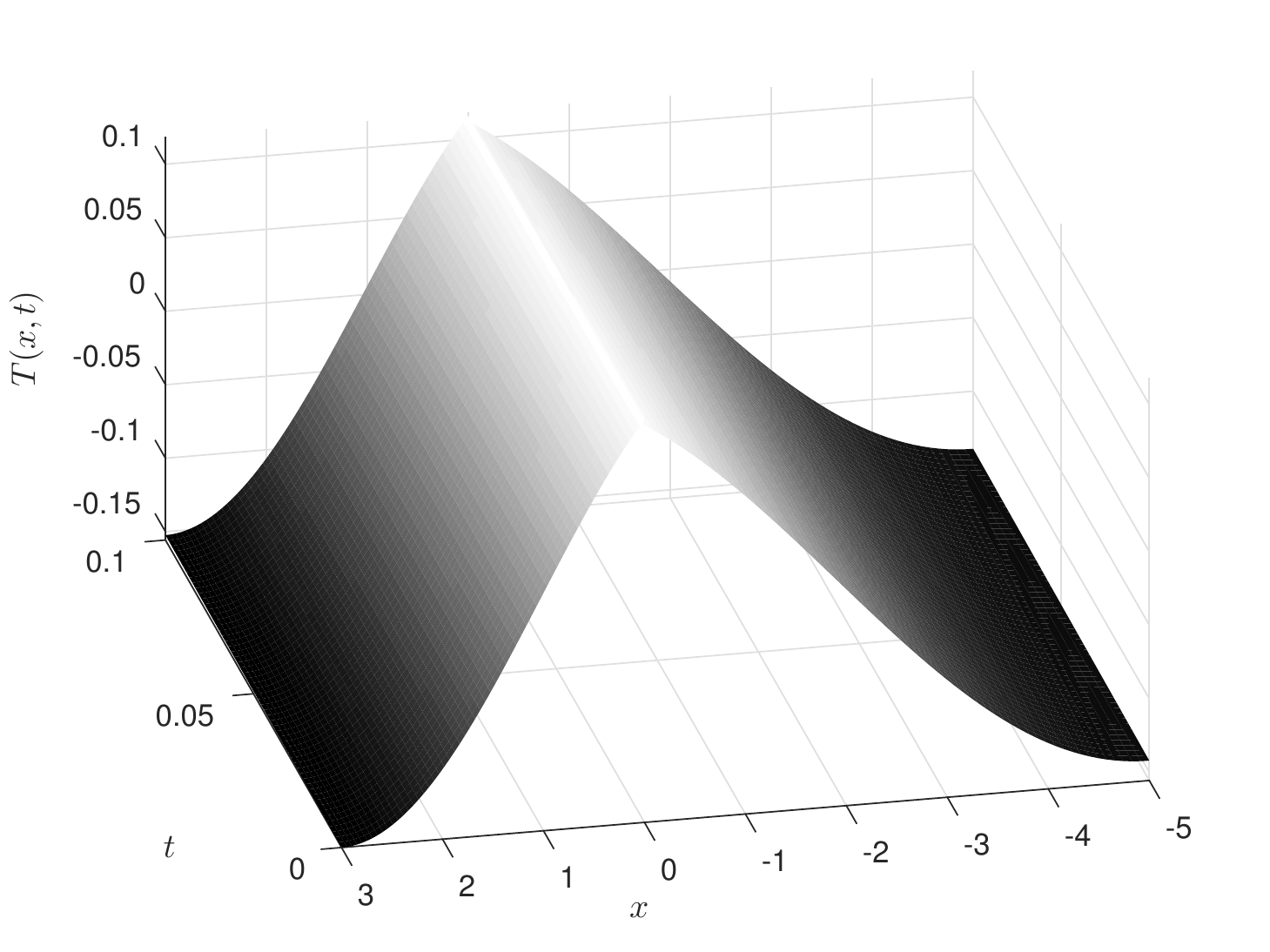}
\par\end{centering}
\caption{The approximate temperature throughout the two-slab system with $\varepsilon=10^{-2}$
in a short-time observation. (Example \ref{subsec:Example-1})\label{fig:Thestability}}
\end{figure}

\begin{figure}
\begin{centering}
\includegraphics[scale=0.8]{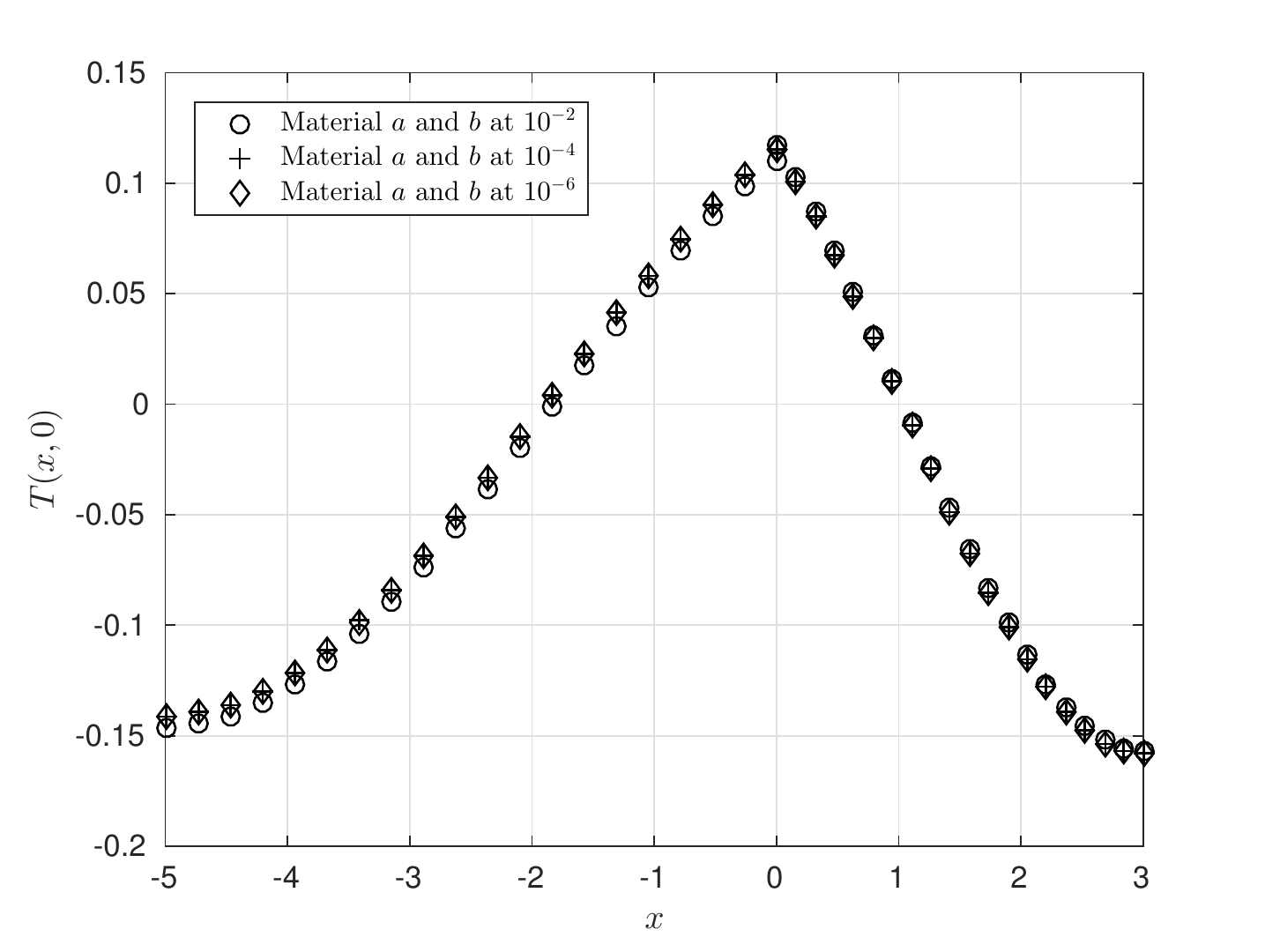}
\par\end{centering}
\caption{The stability illustration of the proposed cut-off projection,
showing the approximate temperature throughout the two-slab system
at $t_{0}=0$ with $\varepsilon\in\left\{ 10^{-2},10^{-4},10^{-6}\right\} $.
(Example \ref{subsec:Example-1})\label{fig:Thestability-1}}
\end{figure}

\subsection{Example 2}\label{subsec:Example-2}

In this example, we assume that the exact initial data are known, albeit of
the unknown exact temperature solution. They are given by piecewise smooth
functions, as follows:
\begin{align*}
T_{b}\left(x,0\right)=\begin{cases}
x, & x\in\left(-\dfrac{b}{2},0\right]\\
-\dfrac{b}{2}, & x\in\left[-b,\dfrac{b}{2}\right]
\end{cases}
&\quad \text{and} \;
T_{a}\left(x,0\right)=\begin{cases}
-\dfrac{a}{2}, & x\in\left(\dfrac{a}{2},a\right],\\
-x, & x\in\left[0,\dfrac{a}{2}\right].
\end{cases}
\end{align*}

In order to observe the approximation $\mathbb{P}_{\varepsilon}T^{\varepsilon}\left(x,t\right)$
to $T\left(x,t\right)$, the following steps are carried out:
\begin{enumerate}
\item Find the final values of $T_{\alpha}\left(x,t_{f}\right)$ from the
forward problem which admits the solution expressed as: for $M\in\mathbb{N}$ (for comparison, we take $M=N_\varepsilon$)
\[
T_{\alpha}\left(x,t\right)=\sum_{n=0}^{M}\tilde{C}_{\alpha n}e^{-\bar{\lambda}_{n}t}\phi_{\alpha n}\left(x\right),
\]

and to find $\tilde{C}_{\alpha n}$, we arrive at
\[
T_{\alpha}\left(x,t_{0}\right)=\sum_{n=0}^{M}\tilde{C}_{\alpha n}\phi_{\alpha n}\left(x\right).
\]
 
\item Obtain the values $T_{\alpha}^{\varepsilon}\left(x,t_{f}\right)$
from those values of $T_{\alpha}\left(x,t_{f}\right)$ by perturbing them, as follows:
\begin{equation*}
T^{\varepsilon}_{\alpha} (x_j,t_f) = T_\alpha(x_j,t_f) +\text{rand}\left(x_{j}\right) \varepsilon = \sum^M_{n=0}\tilde{C}_{\alpha n}e^{-\bar{\lambda}_n t_f}\phi_{\alpha n}(x_j) + \text{rand}\left(x_{j}\right) \varepsilon,
\end{equation*}
with $\left|\text{rand}\left(x_{j}\right)\right|\le (2b)^{-1/4}$ in order to fulfill \eqref{terminalcondition}.

\item {Pass the resulting values to the approximate
solution $\mathbb{P}_{\varepsilon}T^{\varepsilon}\left(x,t\right)$ in accordance with \eqref{approx} to find back the initial data $T_{\alpha}^{\varepsilon}\left(x,t_{0}\right)$.}

\end{enumerate}

Similar steps are also used in Example \ref{subsec:Example-3} with the same type of noise. For our way of choosing $M=N_\varepsilon$ in Step 1, this make things convenient since we can both obtain numerical results for the discrete exact solution and approximate solution with the same dimensions.

\begin{figure}
\begin{centering}
\includegraphics[scale=0.8]{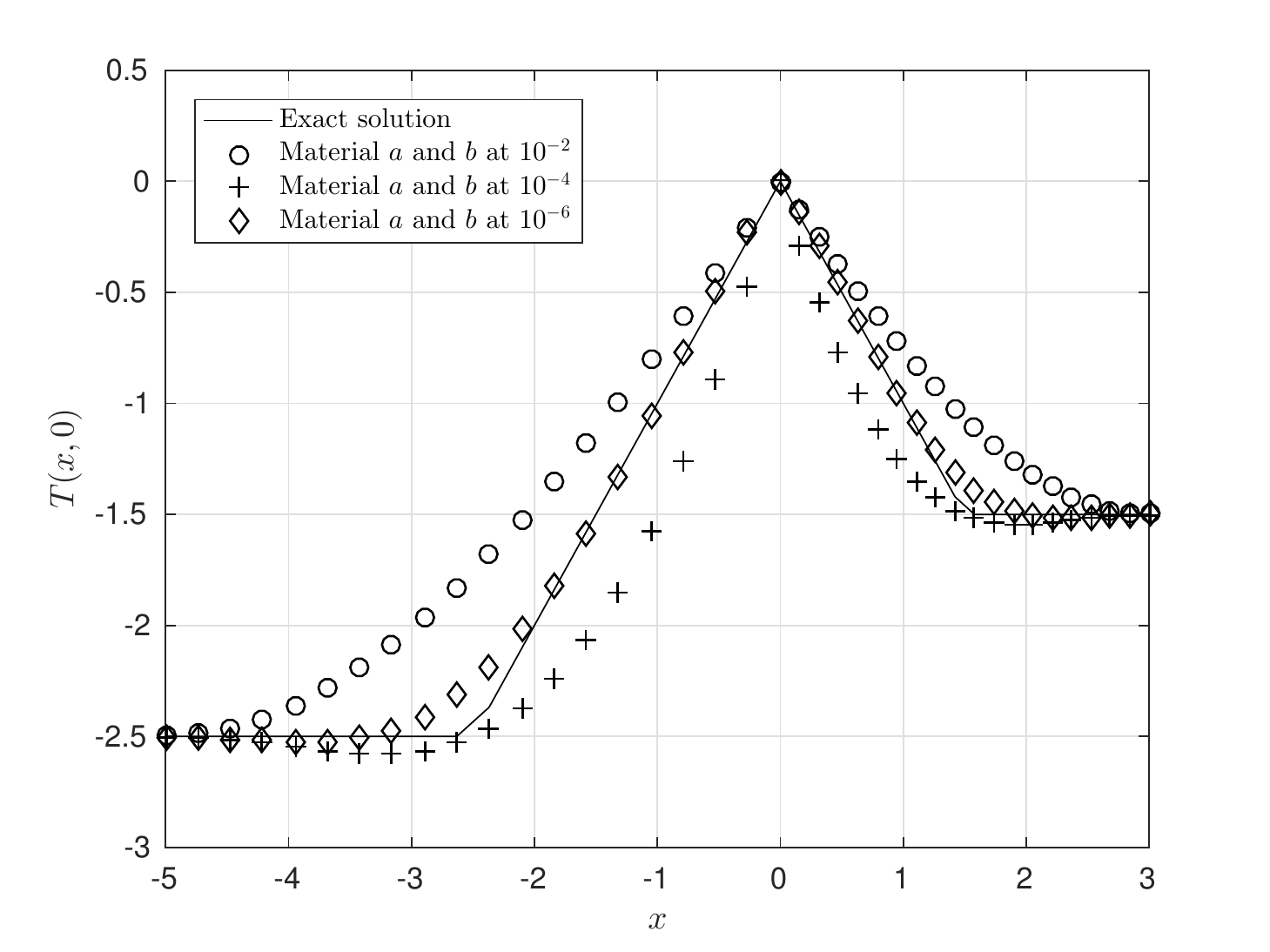}
\par\end{centering}
\caption{The comparison between the (artificial) exact initial data and the computed approximations at $t_{0}=0$ with $\varepsilon\in\left\{ 10^{-2},10^{-4},10^{-6}\right\} $.
(Example \ref{subsec:Example-2})\label{fig:piecewise}}
\end{figure}

\subsection{Example 3}\label{subsec:Example-3}

We finalize this section by testing the example of Parker et al. \cite{Park61}
in which the prescribed initial data are provided by
\[
\begin{cases}
T_{b}\left(x,0\right)=\dfrac{Q}{\rho_{b}c_{b}\sigma}, & x\in\left(-b,0\right),\\
T_{a}\left(x,0\right)=0, & x\in\left(0,a\right),
\end{cases}
\]
where $Q$ is the total heat energy in the initial pulse (cal$\cdot\text{cm}^{-2}$)
and the depth $\sigma$ at the front surface $x=-b$ is performed
as a very small constant compared to the length $a+b$. These two
dimensionless parameters will be selected as $Q=5$ and $\sigma=10^{-3}$.

\begin{figure}
\begin{centering}
\includegraphics[scale=0.8]{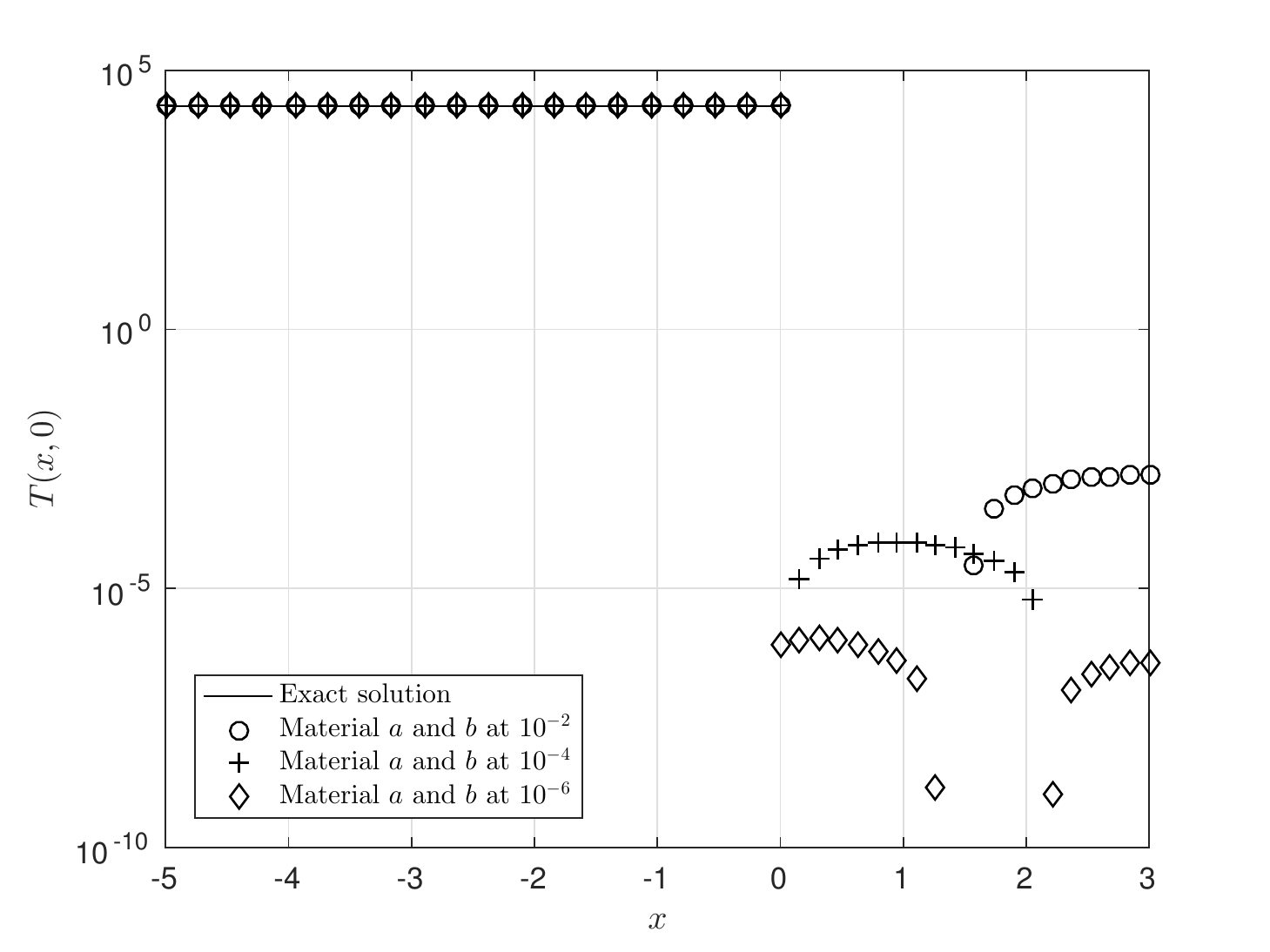}
\par\end{centering}
\caption{The comparison (using log-scale for the $y$-axis) between the (artificial) exact initial data and the computed approximations at $t_{0}=0$ with $\varepsilon\in\left\{ 10^{-2},10^{-4},10^{-6}\right\} $.
(Example \ref{subsec:Example-3})\label{fig:parker}}
\end{figure}

\begin{table}
\begin{centering}
\subfloat[Example \ref{subsec:Example-1}]{\begin{centering}
\begin{tabular}{|c|c|c|c|}
\hline 
$x$ & $\varepsilon=10^{-2}$ & $\varepsilon=10^{-4}$ & $\varepsilon=10^{-6}$\tabularnewline
\hline 
-5 & -0.14623 & -0.14102 & -0.14096\tabularnewline
\hline 
-2.63 & -0.05644 & -0.05122 & -0.05117\tabularnewline
\hline 
1.42 & -0.04657 & -0.04891 & -0.04889\tabularnewline
\hline 
3 & -0.15609 & -0.15843 & -0.15841\tabularnewline
\hline 
\end{tabular}
\par\end{centering}
}
\par\end{centering}
\begin{centering}
\subfloat[Example \ref{subsec:Example-2}]{\begin{centering}
\begin{tabular}{|c|c|c|c|c|}
\hline 
$x$ & $\varepsilon=10^{-2}$ & $\varepsilon=10^{-4}$ & $\varepsilon=10^{-6}$ & Exact\tabularnewline
\hline 
-5 & -2.49806 & -2.50004 & -2.5 & -2.5\tabularnewline
\hline 
-2.63 & -1.83030 & -2.52717 & -2.31534 & -2.5\tabularnewline
\hline 
1.42 & -1.02097 & -1.48 & -1.312359 & -1.42105\tabularnewline
\hline 
3 & -1.49997 & -1.50005 & -1.5 & -1.5\tabularnewline
\hline 
\end{tabular}
\par\end{centering}
}
\par\end{centering}
\begin{centering}
\subfloat[Example \ref{subsec:Example-3}]{\begin{centering}
\begin{tabular}{|c|c|c|c|c|}
\hline 
$x$ & $\varepsilon=10^{-2}$ & $\varepsilon=10^{-4}$ & $\varepsilon=10^{-6}$ & Exact\tabularnewline
\hline 
-5 & 20405.73129 & 20405.72792 & 20405.72792 & 20405.72792\tabularnewline
\hline 
-2.63 & 20405.72961 & 20405.72797 & 20405.72792 & 20405.72792\tabularnewline
\hline 
1.42 & 0.00239 & -2.22E-05 & -3.41E-07 & 0\tabularnewline
\hline 
3 & 0.00159 & 4.53E-05 & 4.26E-07 & 0\tabularnewline
\hline 
\end{tabular}
\par\end{centering}
}
\par\end{centering}
\caption{Numerical results at $t_{0}=0$ with various different points in space.}\label{tab:error}
\end{table}

\subsection{Comments on numerical results}
When approximating the temperature at $t_{0}=0$ for various
amounts of noise $\varepsilon\in\left\{ 10^{-2},10^{-4},10^{-6}\right\} $,
we first observe in Figure \ref{fig:The-instability-illustration} that
the stability of solution ruins drastically for $\varepsilon=10^{-2}$ for Example \ref{subsec:Example-1}. Since in that example we do not have an analytical exact solution, the convergence of the scheme is affirmed by its stability (inspired from the Lax equivalence theorem). Thereby, it can be seen in Figure \ref{fig:Thestability-1} that the approximation by the cut-off projection almost remains
unchanged. For instance, at $(x,t) = (0,0)$ the approximation in the material $a$ for $\varepsilon = 10^{-2}$ yields $0.11787$, whilst the obtained values $\varepsilon = 10^{-4}, 10^{-6}$ are $0.11553
$ and $0.11556$, respectively.  As a result, the stable
approximate solution has been depicted in Figure \ref{fig:Thestability}.

Aside from the first example, the measured final data in the next two examples is established from the artificial initial data by solving the direct problem. Then the approximations can be applied to find back the approximate initial data. 

Figure \ref{fig:piecewise} shows the accuracy of reconstructing the (artificial) exact initial data by the cut-off projection. In fact, we can observe a steady approach of the approximations (nodes) towards the exact solution (black line) as $\varepsilon$ gets from $10^{-2}$ to $10^{-6}$. In line with that, Figure \ref{fig:parker} also confirms a good approximation in the discontinuous case (Example \ref{subsec:Example-3}).

In Table \ref{tab:error}, numerical results with several different spatial points are provided. It is apparent that the numerical solutions in the nonsmooth and discontinuous cases are less stable (accurate) than that of the smooth example. It is simply because of the well-known Gibbs phenomenon, and that the reconstructed data near the nonsmooth and discontinuous points are not well-approximated. Nevertheless, taking into consideration the approximate values in Table \ref{tab:error} at the boundaries as well as the instability in Example \ref{subsec:Example-1} (Figure \ref{fig:The-instability-illustration}), the results in Example \ref{subsec:Example-2} and Example \ref{subsec:Example-3} are acceptable.

\section{Discussion}\label{sec:5}

\subsection{Non-homogeneous case}
We have successfully applied the cut-off method to solve the
heat transfer problem backward in time in a 1D two-slab system with perfect contact. 
This system can be solved similarly when the heat source or sink is
applied to the PDE.  Indeed, let $F_{\alpha}:=F_{\alpha}\left(x,t\right)\in L^{2}\left(-b,a\right)$
represent the heat source or sink.  The system  \eqref{eq:pt1}-\eqref{eq:pt2} now reads
\begin{equation}
\rho_{b}c_{b}\frac{\partial T_{b}}{\partial t}=K_{b}\frac{\partial^{2}T_{b}}{\partial x^{2}} + F_b,\quad-b<x<0,\; t\in J,\label{eq:pt1F}
\end{equation}
\begin{equation}
\rho_{a}c_{a}\frac{\partial T_{a}}{\partial t}=K_{a}\frac{\partial^{2}T_{a}}{\partial x^{2}} + F_a,\quad0<x<a,\; t\in J.\label{eq:pt2F}
\end{equation}

As in Section \ref{sec:2}, the system \eqref{eq:pt1F}-\eqref{eq:pt2F} together with the boundary conditions  \eqref{eq:bien1}-\eqref{eq:bien2} also lead to  the transcendental conditions \eqref{eq:canon}-\eqref{eq:con1}, which allow direct computation of the eigenvalues $\lambda_{an},\lambda_{bn}$ and the eigenfunctions $\phi_{an},\phi_{bn}$ of the Sturm-Liouville equation. The temperature solution (in the finite series truncated computational sense) of the system \eqref{eq:pt1F},\eqref{eq:pt2F}, \eqref{eq:bien1}, \eqref{eq:bien2} has the following form:
\begin{eqnarray}\label{sol:nonhomo}
T(x,t)= \sum_{n=0}^N \left[ C_n e^{\overline{\lambda}_n(t_f-t)} - \int_t^{t_f}D_n(s)e^{\overline{\lambda}_n(s-t)}ds\right] \tilde{\phi}_n(x), \qquad (x,t) \in [-b,a]\times [t_0,t_f],
\end{eqnarray}
for some $N \in \mathbb{N}$. Here,  we follow the notation $\overline{\lambda}_n=\kappa_b\lambda^2_{bn}= \kappa_a \lambda^2_{an}$ in Section \ref{sec:3}, whilst the formula of $\tilde{\phi}_n$ has already been introduced in  \eqref{eq:eigen1}-\eqref{eq:eigen2}. The coefficients $C_n$ can be derived as in Section \ref{sec:3} from the final  temperature condition \eqref{finalstate}, while the algorithm for the time-dependent coefficient $D_{\alpha n}$ is
rather difficult and can only be computed when $C_{\alpha n}$ are
known. First of all, let us return to equations
\eqref{eq:pt1F}-\eqref{eq:pt2F} and then after plugging the representation
\eqref{sol:nonhomo} into each quantity, we see that
\begin{align*}
\frac{\partial T_{\alpha}}{\partial t} & =\sum_{n=0}^{N}\left[D_{\alpha n}\left(t\right)+\bar{\lambda}_{n}\left(\int_{t}^{t_{f}}D_{\alpha n}\left(s\right)e^{\bar{\lambda}_{n}\left(s-t\right)}ds-C_{\alpha n}e^{\bar{\lambda}_{n}\left(t_{f}-t\right)}\right)\right]\phi_{\alpha n},\\
\kappa_{\alpha}\frac{\partial^{2}T_{\alpha}}{\partial x^{2}} & =-\sum_{n=0}^{N}\bar{\lambda}_{n}\left(C_{\alpha n}e^{\bar{\lambda}_{n}\left(t_{f}-t\right)}-\int_{t}^{t_{f}}D_{\alpha n}\left(s\right)e^{\bar{\lambda}_{n}\left(s-t\right)}ds\right)\phi_{\alpha n},\\
\frac{\partial T_{\alpha}}{\partial t} & =\kappa_{\alpha}\frac{\partial^{2}T_{\alpha}}{\partial x^{2}}+\frac{1}{\rho_{\alpha}c_{\alpha}}F_{\alpha},
\end{align*}
where we have used the fundamental differentiation under the
integral sign. Therefore, it follows that
\begin{equation}
\sum_{n=0}^{N}D_{\alpha n}\left(t\right)\phi_{\alpha n}\left(x\right)=\frac{1}{\rho_{\alpha}c_{\alpha}}F_{\alpha}\left(x,t\right)\quad\text{for}\;\alpha\in\left\{ a,b\right\} .\label{eq:2ndrepresentation}
\end{equation}
As the second part of work, fix $t$ in \eqref{eq:2ndrepresentation}
and once again, with the nodes $\left\{ x_{j}\right\} _{j=\overline{0,J_{\alpha}}}$
it enables us to formulate the matrix-formed system $\mathbb{A}\mathbb{X}=\mathbb{B}$
then solve it through. Thereby, we have sufficient values for approximating
the integral ${\displaystyle \int_{t}^{t_{f}}D_{n}\left(s\right)e^{\bar{\lambda}_{n}\left(s-t\right)}ds}$,
and that provides the way to numerically compute the temperature $T\left(x,t\right)$
by discrete nodes on the plane consisting of the time and spatial
variables. 

It is worthy noticing that $F_\alpha$ should satisfy the following condition so that \eqref{eq:2ndrepresentation} is solvable :
\[
\frac{1}{\rho_ac_a}\left< F_a(x,t),\phi_{an}(x)\right>_{L^2(0,a)} = \frac{1}{\rho_bc_b}\left< F_b(x,t),\phi_{bn}(x)\right>_{L^2(-b,0)} \qquad \mbox{for all} \; n \in \mathbb{N}.
\]
The above condition is a necessary condition for $F_\alpha$. We do not aim at finding the sufficient condition to guarantee the existence of solution of the system  \eqref{eq:pt1F}-\eqref{eq:pt2F} in this manuscript.

 \subsection{The bilayer system in 2D}

In this subsection, we generalize the homogeneous problem to two dimensional space. We consider the backward heat transfer problem between two plates: the left plate is the rectangle $\left[-b,0\right]\times\left[0,c\right]$, 
while the right plate is $\left[0,a\right]\times\left[0,c\right]$. The heat distribution in the time interval $\left[t_{0},t_f\right]$
is described by:
\begin{equation}
\rho_{b}c_{b}\frac{\partial T_{b}}{\partial t}=K_{b}\Delta T_{b},\quad\left(x,y\right)\in\left(-b,0\right)\times\left(0,c\right),\; t\in J,\label{eq:pt1-1}
\end{equation}
\begin{equation}
\rho_{a}c_{a}\frac{\partial T_{a}}{\partial t}=K_{a}\Delta T_{a},\quad\left(x,y\right)\in\left(0,a\right)\times\left(0,c\right),\; t\in J.\label{eq:pt2-1}
\end{equation}

Unlike the 
problem \eqref{eq:pt1}-\eqref{eq:pt2}, we have the boundary conditions
on each side of the rectangle and also at the intersection of them,
i.e. $\left(x,y\right)\in\left\{ 0\right\} \times\left[0,c\right]$,
\begin{equation}
\frac{\partial T_{b}}{\partial x}=0\quad\mbox{at}\;x=-b,\quad\frac{\partial T_{a}}{\partial x}=0\quad\mbox{at}\;x=a,\; t\in J, \label{eq:bien1-1}
\end{equation}
\[
\frac{\partial T_{b}}{\partial y}=0\quad\mbox{at}\;\left(x,y\right)\in\left[-b,0\right]\times\left\{ 0,c\right\} ,\quad\frac{\partial T_{a}}{\partial y}=0\quad\mbox{at}\;\left(x,y\right)\in\left[0,a\right]\times\left\{ 0,c\right\} , \; t\in J,
\]
\begin{equation}
T_{b}=T_{a},\quad K_{b}\frac{\partial T_{b}}{\partial x}=K_{a}\frac{\partial T_{a}}{\partial x}\quad\mbox{at}\;\left(x,y\right)\in\left\{ 0\right\} \times\left[0,c\right], \; t\in J.\label{eq:bien2-1}
\end{equation}

Using the same approach as in Section \ref{sec:2}, we can construct the eigen-elements as the solution of 
\begin{equation}
\Delta\phi_{\alpha}\left(x,y\right)+\lambda^{2}_\alpha\phi_{\alpha}\left(x,y\right)=0\quad\text{for }\alpha\in\left\{ a,b\right\} ,\label{eq:4.5}
\end{equation}
 with the assumption that the time dependence of each plates is the
same, and that $\phi_\alpha$ inherits the boundary conditions \eqref{eq:bien1-1}-\eqref{eq:bien2-1}. 

We apply the separation
variable technique for \eqref{eq:4.5} by assuming that $\phi_{\alpha}\left(x,y\right)$
can be expressed as $\phi_{\alpha}\left(x,y\right)=X_{\alpha}\left(x\right)Y_{\alpha}\left(y\right)$.
Then, \eqref{eq:4.5} becomes
\begin{equation}
\dfrac{X_{\alpha}''}{X_{\alpha}}+\lambda^{2}_\alpha=-\dfrac{Y_{\alpha}''}{Y_{\alpha}}=\mu^{2}_\alpha.\label{eq:4.7}
\end{equation}

Direct computation yields that $Y_\alpha (y) = \eta_\alpha \cos(\mu_\alpha y)$, where $\eta_\alpha = \sqrt{\dfrac{2}{c}}$ is a constant and $\mu_\alpha=\mu =\dfrac{m\pi}{c}$ for  $m\in \mathbb{N}$. Since we are solely interested in the case $\lambda^2_\alpha > \mu^2_\alpha$, we can assume that $\nu^{2}_\alpha=\lambda^{2}_\alpha-\mu^2_\alpha$. In the same manner with Section \ref{sec:2}, we obtain
\[
X_{a}=\theta_{a}\cos\left(\nu_{a}\left(x-a\right)\right);\quad X_{b}=\theta_{b}\cos\left(\nu_{b}\left(x+b\right)\right),
\]
where $\theta_a,\theta_b$ are constants. The transcendental condition \eqref{eq:canon} now reads
\begin{equation}\label{eq:5.11}
K_{b}\nu_{b}\sin\left(\nu_{b}b\right)\cos\left(\nu_{a}a\right)+K_{a}\nu_{a}\sin\left(\nu_{a}a\right)\cos\left(\nu_{b}b\right)=0.
\end{equation}

Similarly, a 2D version of condition \eqref{eq:con1} is
\begin{equation}\label{eq:5.12}
\kappa_{a}\left(\nu_{a}^{2}+\mu_{a}^{2}\right)=\kappa_{b}\left(\nu_{b}^{2}+\mu_{b}^{2}\right).
\end{equation}
The temperature solution of the system \eqref{eq:pt1-1}-\eqref{eq:pt2-1} together with the boundary conditions \eqref{eq:bien1-1}-\eqref{eq:bien2-1} is
\[
T_{\alpha}\left(x,y,t\right)=\sum_{m,n\geq0}C_{\alpha mn}e^{\kappa_{\alpha}\left(\nu_{\alpha n}^{2}+\mu_{\alpha m}^{2}\right)\left(t_f-t\right)}X_{\alpha n}\left(x\right)Y_{\alpha m}\left(y\right),
\]
where the coefficient $C_{\alpha mn}$ can be derived from the temperature distribution at the final time:
\[
T_\alpha (x,y,t_f)=\sum_{m,n\geq0}C_{\alpha mn}X_{\alpha n}\left(x\right)Y_{\alpha m}\left(y\right).
\]

For the time being, the algorithm proposed in Section \ref{sec:2} can be applied again in searching for the eigenvalues $\left\{\nu_{\alpha n}\right\}_{n\in \mathbb{N}}$ which are solutions to \eqref{eq:5.11} and \eqref{eq:5.12}. Note from \eqref{eq:5.12} that if the thermal diffusivities of two plates are identical (i.e. $\kappa_{a} = \kappa_{b}$), it turns back to solving the system \eqref{eqn:lambda_b}-\eqref{eqn:lambda_a}.

Applying again the cut-off projection, our regularized solution in the 2D two-slab system is given by
\[
T^{\varepsilon}(x,y,t) = \sum_{m,n \in \Theta(\varepsilon)} C_{mn}e^{\bar{\lambda}_{mn}(t_f - t)}X_{n}(x)Y_{m}(y),
\]
associated with the terminal measured data $T_{\alpha}^{\varepsilon}\left(x,y,t_f\right)$ satisfying
\begin{equation}
\left\Vert T_{a}\left(\cdot,t_{f}\right)-T_{a}^{\varepsilon}\left(\cdot,t_{f}\right)\right\Vert _{L^{2}\left((0,a)\times (0,c)\right)}+\left\Vert T_{b}\left(\cdot,t_{f}\right)-T_{b}^{\varepsilon}\left(\cdot,t_{f}\right)\right\Vert _{L^{2}\left((-b,0)\times(0,c)\right)}\le\varepsilon.\label{terminalcondition2}
\end{equation}

Here, the set of admissible frequencies is defined by $\Theta(\varepsilon):=\left\{ (m,n)\in\mathbb{N}^2:\bar{\lambda}_{mn}\le N_{\varepsilon}\right\} $ with $\bar{\lambda}_{mn} := \kappa_{a}\left(\nu_{an}^{2}+\mu_{am}^{2}\right) = \kappa_{b}\left(\nu_{bn}^{2}+\mu_{bm}^{2}\right)$.

Taking again $\kappa_{\alpha} \in \left\{0.838,0.339\right\}$ and $K_{\alpha} \in \left\{3.42,1.05\right\}$ as in Section \ref{sec:4}, we test the cut-off projection with a simple 2D problem in a normalized square two-slab system, i.e. $a=b=c = 1$, and with $t_f = 0.1$, $\beta = 0.01$, $\gamma = 1$.

In this test, since the values of thermal diffusivities are known, we are allowed to write from \eqref{eq:5.12} that
\begin{equation}\label{eq:5.14}
	\nu_{a} = \sqrt{\frac{\kappa_{b}}{\kappa_{a}} \nu_{b}^2 + \left(\frac{\kappa_{b}}{\kappa_{a}} -1 \right)\mu^2 },
\end{equation}
and thus get the transcendental condition in \eqref{eq:5.11} that
\begin{align}\label{eq:5.15}
& K_{b}\nu_{b} \sin(\nu_{b}b)\cos\left(a\sqrt{\frac{\kappa_{b}}{\kappa_{a}} \nu_{b}^2 + \left(\frac{\kappa_{b}}{\kappa_{a}} -1 \right)\mu^2 } \right) 
\nonumber\\
&+ K_{a}\sqrt{\frac{\kappa_{b}}{\kappa_{a}} \nu_{b}^2 + \left(\frac{\kappa_{b}}{\kappa_{a}} -1 \right)\mu^2 }\sin\left( a \sqrt{\frac{\kappa_{b}}{\kappa_{a}} \nu_{b}^2 + \left(\frac{\kappa_{b}}{\kappa_{a}} -1 \right)\mu^2 }\right)\cos(\nu_{b}b) = 0.
\end{align}

As we can see from \eqref{eq:5.15} , the eigenvalues cannot be computed directly. In order to depict the behavior of those eigenvalues, we demonstrate a simple numerical example by giving an exact smooth form of initial data, viz.
\begin{align*}
T_{b}\left(x,y,0\right) =\cos\left(\pi\left(x+b\right)\right)\cos\left(\pi y\right), &\quad
T_{a}\left(x,y,0\right) =\cos\left(\pi\left(x-a\right)\right)\cos\left(\pi y\right).
\end{align*}

Notice that these are the initial values for the forward problem. Following the steps mentioned in Subsection \ref{subsec:Example-2}, we set the limit of our summation in the expression of Fourier form solution, i.e. $(m,n)\in \Theta(\varepsilon)$ and calculate the eigenvalues according to \eqref{eq:5.14} and \eqref{eq:5.15}. In Figure \ref{fig:example2D-1}, we compute several numerical eigenvalues driven by \eqref{eq:5.14} and \eqref{eq:5.15} by using the built-in function \texttt{fzero} from MATLAB (this has already been introduced in Section \ref{sec:2}).

%
%

In this example we consider the system at $y_0=0$. The algorithm for this 2D case is similar to that for the 1D case (see in Section \ref{sec:4}). Basically, the following steps are carried out:
\begin{enumerate}
\item Find the final values from the forward problem
\[	T_\alpha(x,y_0,t) = \sum_{m,n \in \Theta(\varepsilon)}\tilde{C}_{\alpha mn}e^{-\bar{\lambda}_{n}t}X_{\alpha n}(x)Y_m(y_0), \]
and to find $\tilde{C}_{\alpha mn}$, we solve a linear system $\mathbb{A}\mathbb{X} = \mathbb{B}$ where
\[
\mathbb{A} = 
\begin{bmatrix}
	X_{\alpha 0}(x_0)Y_0(y_0) & X_{\alpha 0}(x_0)Y_1(y_0) & \ldots & X_{\alpha n}(x_0)Y_m(y_0)\\
	X_{\alpha 0}(x_1)Y_0(y_0) & X_{\alpha 0}(x_0)Y_1(y_0) & \ldots & X_{\alpha n}(x_1)Y_m(y_0)\\
	\vdots & \vdots & \ddots & \vdots\\
	X_{\alpha 0}(x_{J_\alpha})Y_0(y_0) & X_{\alpha 0}(x_{J_\alpha})Y_1(y_0) & \ldots & X_{\alpha n}(x_{J_\alpha})Y_m(y_0)
\end{bmatrix} \in \mathbb{R}^{(J_{\alpha}+1)^2},
\]
\[
\mathbb{X} =
\begin{bmatrix}
	\tilde{C}_{\alpha 00}\\
	\tilde{C}_{\alpha 01}\\
	\vdots\\
	\tilde{C}_{\alpha nm}
\end{bmatrix} \in \mathbb{R}^{J_{\alpha} + 1}, \quad
\mathbb{B} =
\begin{bmatrix}
	T_\alpha(x_0,y_0,t_0)\\
	T_\alpha(x_1,y_0,t_0)\\
	\vdots\\
	T_\alpha(x_{J_\alpha},y_0,t_0)
\end{bmatrix}\in \mathbb{R}^{J_{\alpha} + 1}.
\]
In this case we choose $J_\alpha = \left| \Theta(\varepsilon) \right| -1$ where $\left| \Theta(\varepsilon) \right|$ is  the cardinality of  $\Theta(\varepsilon)$.
\item Obtain the values $T_{\alpha}^{\varepsilon}\left(x,y_0,t_{f}\right)$
from those values of $T_{\alpha}\left(x,y_0,t_{f}\right)$ by perturbing them, as follows:
\begin{align*}
T^{\varepsilon}_{\alpha} (x_j,y_0,t_f) = T_\alpha(x_j,y_0,t_f) +\text{rand}\left(x_{j}\right) \varepsilon
& = \sum_{m,n \in \Theta(\varepsilon)}\tilde{C}_{\alpha mn}e^{-\bar{\lambda}_n t_f}X_{\alpha n}(x)Y_m(y_0) + \text{rand}\left(x_{j}\right) \varepsilon,
\end{align*}
with $\left|\text{rand}\left(x_{j}\right)\right|\le ((a+b)c)^{-1/2}$.
\item Perform the proposed method on generated data to find back the initial data $T_{\alpha}^{\varepsilon}\left(x,y_0,t_0\right)$.
\end{enumerate}

\begin{figure}[!h]
	\subfloat[$\lambda_{a}^2$]{\includegraphics[scale=0.6]{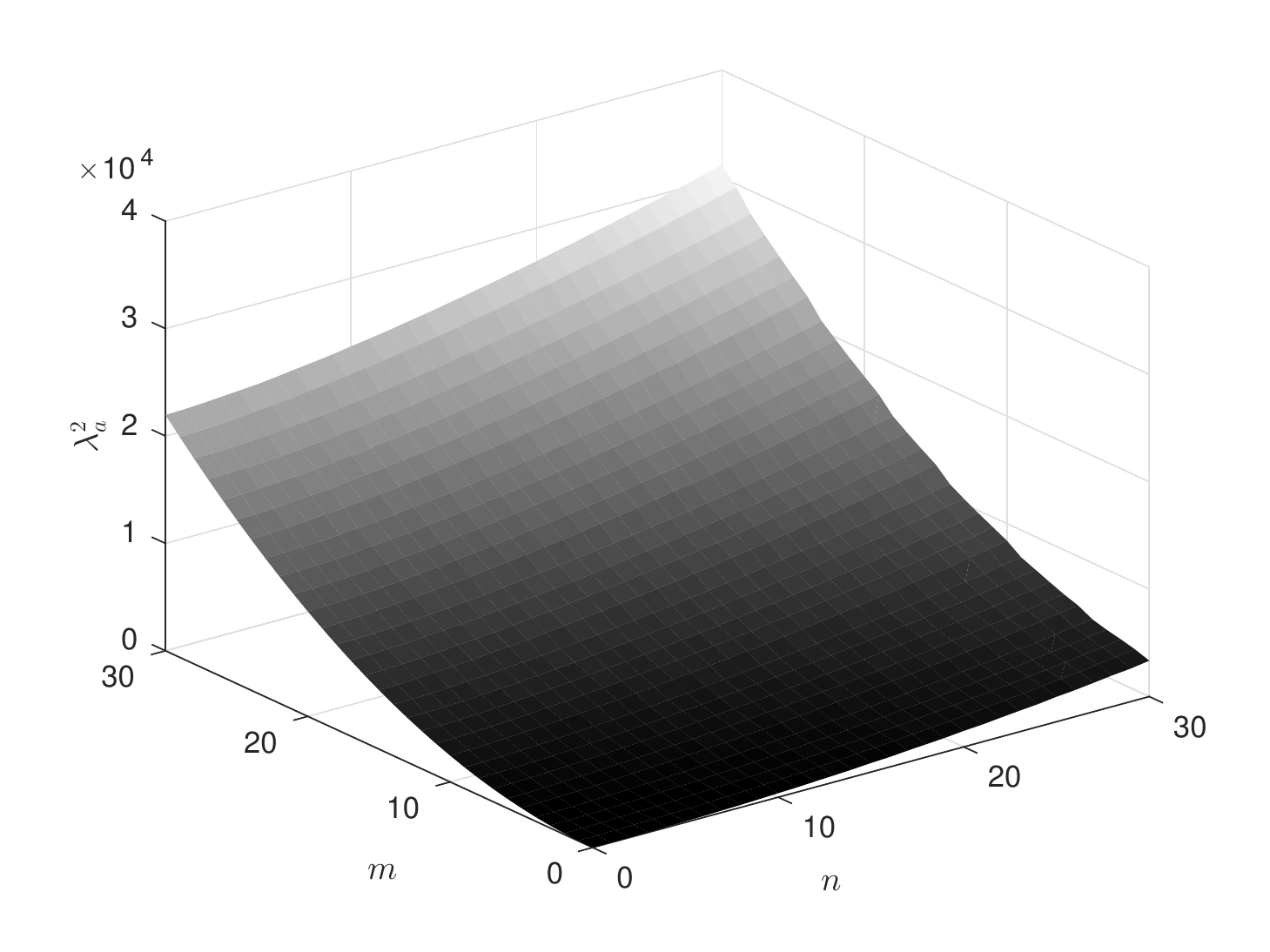}
		
	}\subfloat[$\lambda_{b}^2$]{\includegraphics[scale=0.6]{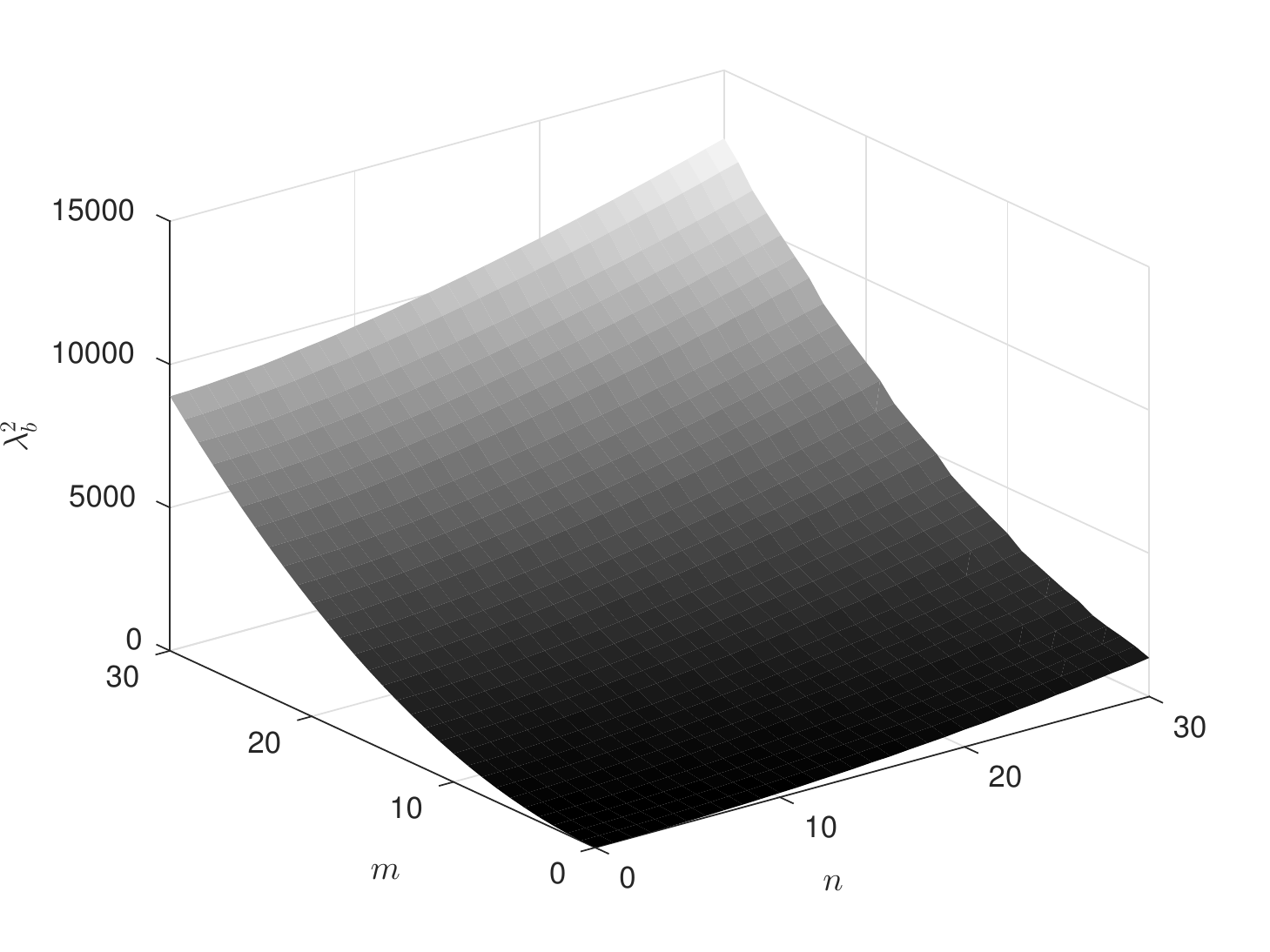}
	}
	\caption{Numerical eigenvalues $\lambda_{\alpha}^{2}=\nu_{\alpha}^2 + \mu_{\alpha}^2$ for the two-slab system in 2D at $\varepsilon = 10^{-6}$.\label{fig:example2D-1}}
\end{figure}

\begin{figure}[!h]
	\includegraphics[scale=0.8]{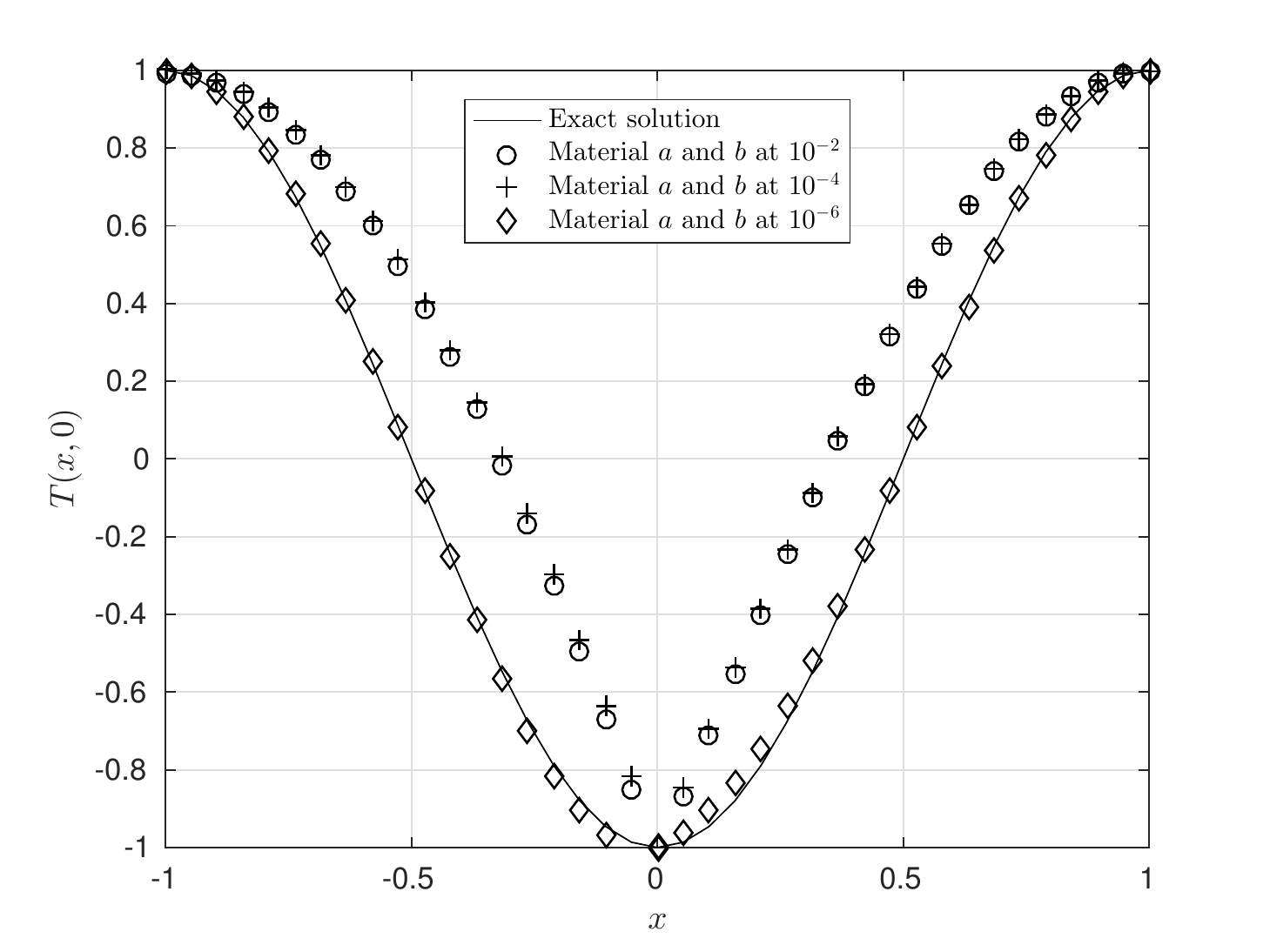}
	\caption{The comparison between the (artificial) exact initial data and the computed approximations at $y_0 = 0, t_{0}=0$ with $\varepsilon\in\left\{ 10^{-2},10^{-4},10^{-6}\right\} $.
	(Example for a two-slab system in 2D).\label{fig:example2D-2}}
\end{figure}

We have drawn in Figure \ref{fig:example2D-2} the numerical results of the exact initial data and the computed approximations based on our proposed cut-off projection at $y_0 = 0$ and at $t_0 = 0$. We can easily observe that these approximations (colored nodes) are closer to the exact initial data (black line) as $\varepsilon$ goes from $10^{-2}$ to $10^{-6}$. This corroborates that our theoretical analysis works well in this 2D example and thus implies the applicability of the proposed method in higher dimensions. 

\subsection{Boundary conditions}
So far, we merely have considered the two-slab problem under idealistic conditions, such as the system has a perfect contact, there is no heat entering or escaping at the boundary of the two slabs. More realistic conditions should be more interesting, for example, an imperfect contact or Robin-type  condition motivated by Newton's Law of Cooling. However, the more complicated boundary conditions, the more severe difficulties of computation we have, especially for the eigen-elements. We shall leave detailed studies of such realistic models for further research.

\subsection*{Acknowledgments}
V.A.K. acknowledges Prof. Errico Presutti (GSSI, Italy) for his invaluable help during the PhD years. The authors wish to express their sincere thanks to the anonymous referees and the associate editors for many constructive comments leading to the improved versions of this paper. This work is supported by the Institute for Computational Science and Technology (Ho Chi Minh city) under project named ``Some ill-posed  problems for partial differential equations''.

\bibliographystyle{plain}
\bibliography{mybib}

\end{document}